\documentclass{amsart}
\usepackage[mathscr]{eucal}
\usepackage{amsmath,amssymb,hyperref}
\usepackage{graphics}
\usepackage{multirow}
\usepackage[all]{xy}
\usepackage{enumerate}

\newtheorem{thm}{Theorem}[section]
\newtheorem{THM}{Theorem}

\newtheorem{cor}[thm]{Corollary}
\newtheorem{prop}[thm]{Proposition}

\newtheorem{lemma}[thm]{Lemma}

\theoremstyle{definition}

\newtheorem{definition}[thm]{Definition}
\newtheorem{remark}[thm]{Remark}

\newcommand{\con}{N_\F^*}
\newcommand{\Iadm}{{\mathcal I}^\varepsilon}
\newcommand{\Ieucl}{{\mathcal I}^0}
\newcommand{\Ihyp}{{\mathcal I}^1}
\newcommand{\ladm}{{\mathcal I}_{d\log}^\varepsilon}
\newcommand{\leucl}{{\mathcal I}_{d\log}^0}

\DeclareMathOperator{\ddbar}{\partial\overline{\partial}}
\DeclareMathOperator{\dbar}{\overline{\partial}}

\DeclareMathOperator{\Log}{Log}

\def\R{\mathbb R}
\def\C{\mathbb C}

\def\Q{\mathbb Q}
\def\F{\mathcal F}
\def\D{\mathbb D}

\input xy
\xyoption{all}

\usepackage[active]{srcltx}

\begin{document}
 
 \title[On the structure of codimension 1 foliations ] {On the structure of codimension 1 foliations with pseudoeffective conormal bundle.}
 \author[FR\'ED\'ERIC TOUZET]{FR\'ED\'ERIC TOUZET$^1$}
\thanks{$^1$ IRMAR, Campus de Beaulieu 35042 Rennes Cedex, France, frederic.touzet@univ-rennes1.fr}


\begin{abstract} Let $X$ a projective manifold equipped with a codimension $1$ (maybe singular) distribution whose conormal sheaf is assumed to be pseudoeffective. By a theorem of Jean-Pierre Demailly, this distribution is actually integrable and thus defines a codimension $1$ holomorphic foliation $\F$.  We aim at describing the structure of such a foliation, especially in the non abundant case: It turns out that $\F$ is the pull-back of one of the "canonical foliations"  on a Hilbert modular variety. This result remains valid for ``logarithmic foliated pairs''.

\end{abstract}
\maketitle



\setcounter{tocdepth}{1}
\sloppy
\tableofcontents

\section{Introduction}

In this paper, we are dealing with special class of codimension $1$ holomorphic foliations. To this goal, let us introduce some basic notations. 

Our main object of interest consists of a pair $(X,\mathcal D)$ where $X$ stands for a connected complex manifold (which will be fairly quickly assumed to be K\"ahler compact) equipped with a codimension $1$ holomorphic (maybe singular) distribution $\mathcal D=\mbox{Ann}\ \omega\subset TX$ given as the annihilator subsheaf of a twisted holomorphic non trivial one form 

$$\omega\in H^0(X,\Omega_X^1\otimes L).$$

Without any loss of generalities, one can assume that $\omega$ is surjective in codimension $1$, in other words that the vanishing locus $\mbox{Sing}\  \omega$ has codimension at least $2$.

The line bundle $L$ is usually called the {\bf normal bundle} of $\mathcal D$ and denoted by $N_{\mathcal D}$. Here, we will be particularly interested in studying its dual  $N_{\mathcal D}^*$, the {\bf conormal bundle} of $\mathcal D$ in the case this latter carries some ``positivity'' properties.

We will also use the letter $\mathcal F$ (as foliation) instead of $\mathcal D$ whenever integrability holds, i.e: $[\mathcal D,\mathcal D]\subset \mathcal D$.

Assuming now that $X$ is {\bf compact K\"ahler}, let us review some well known situations where integrability automatically holds.
\vskip 5 pt

{\bf $\mathcal D=\mathcal F$} as soon as:

\begin{enumerate}
 \item $N_{\mathcal D}^*=\mathcal O (D)$ where $D$ is an integral effective divisor. This is related to the fact that $d\omega=0$ (regarding $\omega$ as an holomorphic one form with $D$ as zeroes divisor).
\item $\mbox{Kod}\ (N_{\mathcal D}^*)\geq 0$ where $\mbox{Kod}$ means the Kodaira dimension. Actually, in this situation, one can get back to the previous one passing to a suitable branched cover.
\end{enumerate}

In this setting, it is worth recalling the classical result of {\bf Bogomolov Castelnuovo De Franchis} wich asserts that one has always $\mbox{Kod}\ (N_{\mathcal D}^*)\leq 1$ and when equality holds, $\mathcal D=\mathcal F$ is an holomorphic fibration over a curve.
\vskip 5 pt

In fact, integability still holds under weaker positivity assumptions on $N_{\mathcal D}^*$, namely if this latter is only supposed to be {\bf pseudoeffective} ({\it psef}). This is a result obtained by Demailly in \cite{de}.

For the sequel, it will be useful to give further details and comments on this result.

First, the pseudoeffectiveness property can be translated  into the existence of a singular metric $h$ on $N_{\mathcal D}^*$ with plurisubharmonic ({\it psh}) local weights. This means that $h$ can be locally expressed as 

$$h(x,v)={|v|}^2e^{-2\varphi (x)}$$

\noindent where $\varphi$ is a {\it psh} function.

The curvature form $T=\frac{i}{\pi}\partial\overline{\partial}\varphi$ is then well defined as a closed positive current representing the real Chern class $c_1(N_{\mathcal D}^*)$ and conversely any closed $(1,1)$ positive current $T$ such that $\{T\}=c_1(N_{\mathcal D}^*)$ (here, braces stand for ``cohomology class'') can be seen as the curvature current of such a metric.

What actually Demailly shows is the following vanishing property:

$$ \nabla_{h^*}\omega=0.$$

More explicilely, $\omega$ is closed with respect to the Chern connection attached to $N_{\mathcal D}$ endowed with the dual metric $h^*$.

This somehow generalizes the well known fact that every holomorphic form on compact K\"ahler manifold is closed.

Locally speaking, this equality can be restated as 
\begin{equation}\label{connectionclosed}
2\partial\varphi\wedge\omega=-d\omega. 
\end{equation}
\noindent if one looks at $\omega$ as a suitable local defining one form of the distribution.

Equality (\ref{connectionclosed}) provides three useful informations:

The first one is the sought integrability: $\omega\wedge d\omega=0$, hence $\mathcal D=\mathcal F$.

One then obtains taking the $\overline{\partial}$ operator, that $T$ is {\bf $\mathcal F$-invariant} (or invariant by holonomy of the foliation), in other words, $T\wedge \omega=0$.

One can also derive from (\ref{connectionclosed}) the closedness of $\eta_T$, where $\eta_T$ is a globally defined positive $(1,1)$ form which can be locally expressed as 

    $$\eta_T=\frac{i}{\pi}e^{2\varphi}\omega\wedge\overline{\omega}.$$

In some sense, $\eta_T$ defines the dual metric $h^*$ on $ N_{\mathcal D}$ whose curvature form is $-T$ and is canonically associated to $T$ up to a positive multiplicative constant.

Note that, as well as $T$, the form $\eta_T$ has no reasons to be smooth. Nevertheless, it is well defined as a current and taking its differential makes sense in this framework.

By closedness, one can also notice that $\eta_T$ is indeed a {\bf closed positive $\mathcal F$ invariant current}.
\vskip 5 pt

Now, the existence of two such invariant currents, especially $\eta_T$ which has locally bounded coefficients, may suggest that this foliation has a nice behaviour, in particular from the dynamical viewpoint, but also, as we will see later, from the algebraic one. Then it seems reasonnable to describe them in more details.
\vskip 5 pt

As an illustration, let us give the following (quite classical) example wich also might be enlightening in the sequel.

Let $S=\frac{\D\times\D}{\Gamma}$ a surface of general type uniformized by the bidisk. Here, $\Gamma\subset {\mbox{Aut}\ \D}^2$ is a cocompact irreducible and torsion free lattice wich lies in the identity component of the automorphism group of the bidisk.

The surface $S$ is equipped with two ``tautologically'' defined minimal foliations by curves ${\mathcal F}_h$ and ${\mathcal F}_v$ obtained respectively by projecting the horizontal and vertical foliations of the bidisk (recall that ``minimal'' means that every leaf is dense and here, this is due to the lattice irreducibility).

Both of these foliations have a {\it psef} and more precisely a semipositive conormal bundle. For instance, if one considers $\mathcal F={\mathcal F}_h$, a natural candidate for $\eta_T$ and $T$ is provided by the transverse Poincar\'e's metric $\eta_{\mbox{Poincar\'e}}$, that is the metric induced by projecting the standard Poincar\'e's metric of the second factor of $\D\times\D$.

In this example, it is staightforward to check that $T=\eta_T$ is $\mathcal F$ invariant. We have also $T\wedge T=0$. This implies that the numerical dimension $\nu (N_{\mathcal D}^*)$ is equal to one. By minimality, and thanks to Bogomolov's theorem this equality certainly does not hold for $\mbox{Kod}\ N_{\mathcal D}^*$. In fact, this is not difficult to check that this latter is negative.

Maybe this case is the most basic one where abundance does not hold for the conormal. One can then naturally ask whether this foliation is in some way involved as soon as the abundance principle fails to be true. It turns out that (at least when the ambient manifold $X$ is projective) the answer is yes, as precised by the following statement: 

\begin{THM}\label{thprincipal}
 Let $(X,\mathcal F)$ be a foliated projective manifold by a codimension $1$ holomorphic foliation $\mathcal F$.  Assume moreover that 
$N_{\mathcal F}^*$ is pseudoeffective and $\mbox{Kod}\ N_{\mathcal F}^*=-\infty$, then one can conclude that $\mathcal F={\Psi}^*\mathcal G$ where $\Psi$ is a morphism of analytic varieties between $X$ and the quotient $\frak{H}={{\D}^n}/{\Gamma}$ of a polydisk, ($n\geq 2)$, by an irreducible lattice $\Gamma\subset {(\mbox{Aut}\ \D)}^n$ and $\mathcal G$ is one of the $n$ tautological foliations on $\frak{H}$.
\end{THM}
\vskip 5 pt

One can extend this result for ``logaritmic foliated pairs'' in the following sense:

 \begin{THM}\label{thprincipallog}
 The same conclusion holds replacing the previous assumptions  on $N_{\mathcal F}^*$ by the same assumptions on the logarithmic conormal bundle $N_{\mathcal F}^*\otimes \mathcal O(H)$ where $H=\sum_iH_i$ is reduced and normal crossing divisor, each component of which being invariant by the foliation and replacing $\frak H$ by its Baily-Borel compactification ${\overline{\frak H}}^{BB}$.

\end{THM}

\begin{remark}
 In the logaritmic case, compactification is needed because, unlike to the non logarithmic setting, there may exist among the components of $H$ some special ones arising from the possible existence of cusps in $\frak H$.
\end{remark}

Let us give the outline of the proof. For the sake of simplicity, we will take only in consideration the non logarithmic situation corresponding to theorem \ref{thprincipal}. The logarithmic case follows {\it mutatis mutandis} the same lines althought we have to overcome quite serious additional difficulties.

The first idea consists in performing and analysing the divisorial Zariski decomposition of the pseudoeffective class $c_1(N_{\mathcal F}^*)$ using that we have at our disposal the two positive closed invariant currents $T$, $\eta_T$ and their induced cohomology class $\{T\}$ and $\{\eta_T\}$.

Actually, the existence of such currents yields strong restriction on the Zariski decomposition. Without entering into details, this basically implies that this latter shares the same nice properties than the classical two dimensional Zariski decomposition. 

Once we have obtained the structure of this decomposition, one can solve by a fix point method (namely the Shauder-Tychonoff theorem applied in a suitable space of currents) the equation 

\begin{equation}\label{equa-1}
T=\eta_T+[N]
\end{equation}

\noindent
where $[N]$ is some ``residual'' integration current. More precisely, among the closed positive currents representing $c_1( N_{\mathcal F}^*)$, there exists one (uniquely defined) wich satisfies the equality ($\ref{equa-1}$) (with $\eta_T$ suitably normalized).

This equality needs some explanations. Indeed, it really holds when the positive part in the Zariski decomposition is non trivial, that is when $N_{\mathcal F}^*$ has positive numerical dimension. From now on, we focus only on this case, regardless for the moment of what occurs for vanishing Kodaira dimension. 

We also have to precise what is $N$: it is just the negative part in the Zariski decomposition and it turns out that $N$ is a $\Q$ effective $\mathcal F$ invariant divisor. 

Then, outside $\mbox{Supp}\ N$, near a point where the foliation is assumed to be regular and defined by $dz=0$, ($\ref{equa-1}$) can be locally expressed as 

$${\Delta}_z\varphi (z)=e^{2\varphi (z)}$$

\noindent where $\varphi$ is a suitably chosen local potential of $T$ depending only of the transverse variable $z$ and $\Delta_z$ is the Laplacean with respect to the single variable $z$. This is the equation giving metric of constant negative curvature, wich means that $\mathcal F$ falls into the realm of the so called ``transversely projective'' foliations and more accurately the transversely hyperbolic ones. Equivalently, those can be defined by a collection of holomorphic local first integral valued in the disk such the glueing transformation maps between these first integrals are given by disk automorphisms. In that way, it gives rise to a representation 

$$\rho:\pi_1 (X\setminus \mbox{Supp}\ N)\rightarrow \mbox{Aut}\ \D.$$

\begin{remark}
 We need to remove $\mbox{Supp}\ N$ because the tranversely hyperbolic structure may degenerate on this hypersurface (indeed, on a very mildly way). 
\end{remark}

Let $G$ be the image of the representation: it is the monodromy group of the foliation and faithfully encodes the dynamical behaviour of $\mathcal F$. More precisely, one can show that there are two cases to take into consideration:

Either $G$ is a cocompact lattice and in that case, $\mathcal F$ is an holomorphic fibration, $\mbox{Kod}\ N_{\F}^*=1$ and coincide with the numerical dimension.

Either $G$ is dense (w.r.t the ordinary topology) in $\mbox{Aut}\ \D$ and the foliation is quasiminimal (i.e: all leaves but finitely many are dense). Moreover, $\mbox{Kod}\  N_{\F}^*=-\infty$. Of course, we are particularly interested in this last situation. 

To reach the conclusion, the remaining arguments heavily rely on a nice result of Kevin Corlette and Carlos Simpson (see \cite{corsim}) concerning rank two local systems on quasiprojective manifolds, especially those with Zariski dense monodromy in $SL(2,\C)$. 

Roughly speaking, their theorem gives the following
alternative:

\begin{enumerate}
\item
Either such a local system comes from a local system on a Deligne-Mumford curve $\mathcal C$.
\item
Either it comes from one of the taulogical rank two local systems on a what the authors call a ``polydisk Shimura DM stack'' $\frak H$ (this latter can be thought as some kind of moduli algebraic space uniformized by the polydisk, in the same vein than Hilbert Modular varieties).
\end{enumerate}

\begin{remark}
 In the setting of transversely hyperbolic foliations, as we do not really deal with rank 2 local but rather with projectivization of these, one can get rid of the DM stack formalism, only considering the target spaces $\mathcal C$ or $\frak H$ with their ``ordinary''orbifold strucure (wich consists of finitely many points). Moreover, in the full statement of Corlette-Simpson, one needs to take into account  some quasiunipotency assumptions on the monodromy at infinity. Actually, these assumptions are fulfilled in our situation and are related to the rationality of the divisorial coefficients in the negative part $N$.

\end{remark}

 This paper is a continuation of the previous work \cite{to} where one part of the strategy outlined above has been already achieved (namely the construction of special metrics transverse to the foliation).

Note also that the desription of transversely projective foliations on projective manifolds was recently completed in \cite{lptbis}. However, there are several significant differences between our paper and {\it loc.cit} (and in fact, the techniques involved in both papers are not the same):

\begin{itemize}
 \item In our setting, the transverse projective structure is not given a priori,

\item \cite{lptbis} also deals with irregular transversely projective foliations (ours turn out to be regular singular ones).

\item  When the monodromy associated to the transverse projective structure is Zariski dense in $PSL(2,\C)$, the factorization theorem D, item (2) and (3) of \cite{lptbis} can be made more precise for the class of foliations considered here. Namely, pull-back of a Riccati foliation over a curve (item (2)) is replaced by pull-back of a foliation by points on a curve, that is $\F$ is a fibration, and analogously, item (3) can be reformulated changing ``pull-back between representations'' by ``pull-back between foliations'', as precised by theorems \ref{thprincipal} and  \ref{thprincipallog}  above.
\end{itemize}
 
\vskip 10 pt
{\bf Acknowledgements}: It's a pleasure to thank Dominique Cerveau, Frank Loray and Jorge Vit\'orio Pereira for valuable and helpful discussions.

\section{Positive currents invariant by a foliation}

In this short section, we review some definitions/properties (especially intersection ones) of positive current which are invariant by holonomy of a foliation. They will be used later.

\subsection{Some basic definitions}
\noindent Let $\F$ a codimension $1$ holomorphic foliation on a complex manifold $X$ defined by $\omega\in H^0(X,\Omega_X^1\otimes L)$.

\begin{definition} (see also \cite{brbir}, \cite{to}).
Let $T$ a positive  current of bidegree $(1,1)$ defined on $X$. $T$ is said to be {\bf invariant by $\mathcal F$}, or $\mathcal F$-invariant, or invariant by holonomy of $\F$ whenever
\begin{enumerate}
\item $T$ is closed
\item $T$ is directed by the foliation, that is $T\wedge\omega=0$.

\end{enumerate}

\end{definition}

Notice that when $T$ is a non vanishing smooth semi-positive $(1,1)$ form, we fall into the realm of the so-called {\bf transversely hermitian foliations}.

\begin{definition} (see\cite{to}).
Let $x\in X$ a singular point of $\mathcal F$. We say that $x$ is an elementary singularity if there exists near $x$ a non constant first integral of $\mathcal F$ of the form 
$$f=f_1^{\lambda_1}...f_p^{\lambda_p}$$
where $f_i\in {\mathcal O}_{X,x}$ and $\lambda_i$ is a positive real number .

\noindent Equivalently, $\frac{df}{f}$ is a local defining logarithmic form of $\mathcal F$ with positive residues.
\end{definition}

\begin{definition}
Let $T$ a closed positive $(1,1)$ current on $X$. Recall that $T$ is {\bf residual } is the Lelong number $\lambda_H (T)$ of $T$ along any hypersurface $H$ is zero.

\end{definition}

Let $X$ be compact K\"ahler, and $T$ a $(1,1)$ positive current on $X$ and $\{T\}\in H^{1,1}(X,\R)$ its cohomology class . By definition, this latter is a {\it pseudoeffective}  class and admits, as such, a {\bf divisorial Zariski decomposition} ( a fundamental result due to \cite{bo}, \cite{na}, among others):

$$\{T\}={N} + Z$$

\noindent where $N$ is the "negative part",  an  $\R$ effective divisor, whose support is a finite union of prime divisor $N_1,.. N_p$ such that the family $\{N_1,..,N_p\}$ is exceptional (in the sense of \cite{bo})  and $Z$, the "positive part" is {\it nef} in codimension $1$ (or modified {\it nef}). This implies in particular that the restriction of $Z$ to any hypersurface is still a pseudoeffective class.

(For a precise definition and basic properties of Zariski decomposition useful for our purpose, see \cite{to} and references therein).

\subsection{Intersection of $\mathcal F$-invariant positive currents}

Let $X$ is a K\"ahler manifold equipped with a codimension $1$ holomorphic foliation $\F$ wich carries a $(1,1)$ closed positive $T$ current invariant by holonomy. 

\begin{prop}\label{potconst}
Assume that near every point, $T$ admits a local potential locally constant on the leaves.
 Let $S$ be an other closed positive current invariant by $\mathcal F$, then $\{S\}\{T\}=0$
 \end{prop}


\begin{proof}

Let  $(U_i)_{i\in I}$ an open cover of $X$ such that $T=\frac{i}{\pi}\partial\overline{\partial}\varphi_i$, $\varphi_i$ {\it psh} locally constant on the leaves.. On intersections, one has $\partial\varphi_i-\partial\varphi_j=\omega_{ij}$ where $\omega_{ij}$ is a holomorphic one  form  defined on $U_i\cap U_j$  vanishing on the leaves (i.e a local section of the conormal sheaf $N_\F^*$ considered as an invertible subsheaf of $\Omega_X^1$). Let $\eta$ a $(1,1)$ closed smooth form such that $\{\eta\}=\{T\}$. One can then write on each $U_i$ (up refining the cover) $\eta=\frac{i}{\pi}\partial\overline{\partial}u_i$, $u_i\in{\mathcal C}^\infty (U_i)$, such that $\partial u_i-\partial u_j=\omega_{ij}$. As a consequence, the collection of $\frac{i}{\pi}\partial u_i\wedge S$, with $S$ any closed invariant current, gives rise to a globally well defined $(1,2)$ current whose differential is $\eta\wedge S$. This proves the result.
\end{proof}

\begin{remark}\label{extensiontocodimension3}
 One can slightly improve this statement replacing $X$ by $X\setminus S$ where $S$ is an analytic subset of codimension $\geq 3$ (it will be useful in the sequel).

To this goal, we use that $H^1(U_i\setminus S, N_{\F}^*)=0$ (for a suitable cover). This  means that one can find on each $U_i$  an holomorphic one form $\Omega_i$ such that $\partial \varphi_i-\partial\varphi_j=\omega_{\ij}+\Omega_i-\Omega_j$. Indeed, take an open cover $(V_k)$ of $U_i\setminus S$ such that $\partial\varphi_i=\partial\varphi_k+ \xi_k$ where $\varphi_k$ is {\it psh} and $\xi_k$ an holomorphic one form both defined on $V_k$.  By vanishing of the cohomology,  $\xi_k-\xi_l=\omega_k-\omega_l$ where $\omega_k$ is an holomorphic form vanishing on the leaves .  Then the $\xi_k-\omega_k$'s glue together on $U_i\setminus S$ into a holomorphic one form whose extension through $S$ is precisely $\Omega_i$.  Taking $\eta$ and $S$ as previously, one obtains that the current defined locally by $(\partial u_i-\Omega_i)\wedge S$ is globally defined. Then, the same conclusion follows taking the $\overline{\partial}$ differential instead of the differential and applying the $\partial\overline{\partial}$ lemma.\qed

\end{remark}

\begin{prop}\label{intersectionresidual}
 Assume that the singular points of $\mathcal F$ are all of {\it elementary type} (as defined above).Let $S,T$  be $\F$-invariant positive currents. Assume moreover that $S$ is residual  (i.e, without any Lelong's numbers in codimension $1$), then the local potentials of $S$ can be chosen to be constant on the leaves and in particular

\begin{enumerate}
\item $\{S\}\{T\}=0$
\item $\{S\}$ and $\{T\}$ are colinear assuming in addition that $T$ is also residual.
\end{enumerate}
\end{prop}

\begin{proof} By proposition \ref{potconst}, it suffices to show that  one can choose the local potentials of $T$ to be locally constant on the leaves to prove the first item.

Let $x\in \mbox{Sing}\ \F$ and $F$ be an elementary first integral of $X$ near $x$. One can assume that 

$$F=\prod_{i=1}^p {f_i}^{\lambda_i}$$

\noindent with $\lambda_1=1$. Following the main result of (\cite{pa2}), the (multivaluate) function $F$ has connected fibers on $U\setminus 	{Z(F)}$ where $U$ is a suitable arbitrarily small open neighborhood of $x$ and $Z(F)=\{\prod f_i=0\}$. Let $\mathcal T\subset U$ an holomorphic curve transverse to $\mathcal F$ with cuts out the branch $f_1=0$ in $x_1$, a regular point of $\F$ and such that ${f_1}_{|{\mathcal T}}$ sends biholomorphically $\mathcal T$ onto a disk ${\mathbb D}_r=\{|z| <r\}$. Let $G$ be the closure in $S^1=\{z=1\}$ of the group generated by $\{e^{2i\pi\lambda_k}\},\  k=1,...,p$; thus, either $G$ is the whole $S^1$ or is finite.

Up shrinking $U$, one can assume that, on $U$, $T=\frac{i}{\pi}\partial\overline{\partial}\varphi$ where $\varphi$ is {\it psh}. Let $\varphi_1=\varphi_{|\mathcal T}$ and $T_1=\frac{i}{\pi}\ddbar{\varphi_1}$. 

Set $$H=\{h\in\mbox{Aut}(\mathcal T)|h(z)=gz, g\in G\}.$$

\noindent (here, and in the sequel we  identify $\mathcal T$ and ${\mathbb D}_r$).
 The fibers connectedness property remains clearly valid if one replaces $U$ by the saturation of
 $\mathcal T$ by $\F$ in $U$ (still called $U$). 
 
 Thus  $T_1$ is $H$-invariant, that is $h^*T_1=T_1$ for very $h\in H$. One also obtain a subharmonic function $\psi_1$    $H$-invariant when averaging $\varphi_1$  with respect to the Haar measure $dg$ on the compact group $G$ (\cite{ra}, th 2.4.8):

$$\psi_1(z)= \int_G \varphi_1(gz)dg$$ 

\noindent Note that $\frac{i}{\pi}\ddbar{\psi_1}=T_1$; then, by $H$ invariance, $\psi_1$ uniquely extends on $U$ as a ${\it psh}$ function $\psi$ locally constant on the leaves and such that

$$T=\frac{i}{\pi}\partial\overline{\partial}{\psi}$$

Indeed, this equality obviously holds on $U\setminus Z(F)$ (where the foliation is regular)  and then on $U$ as currents on each side  do not give mass to $Z(F)$. This proves the first part of the proposition.

The assertion $(2)$ is a straightforward consequence of the Hodge's signature theorem if one keeps in mind that ${\{S\}}^2={\{T\}}^2= \{S\}\{T\}=0$.

\end{proof}
\begin{remark}\label{elementaryinsupport}
  Actually, the same conclusion (and proof) holds under the weaker assumption that every singular point of $\F$ contained in $\mbox{Supp}\ S\cap\mbox{Supp}\ T$ is of elementary type, instead of requiring all the singular points to be elementary.
\end{remark}

\section{Foliations with pseudoeffective conormal bundle} \label{psefcon} 

Throughout this section, $X$ is a compact K\"ahler manifold carrying a codimension $1$ (maybe singular) holomorphic foliation whose conormal bundle $N_\mathcal F^*$ is pseudoeffective ({\it psef} for short). The study of such objects has been undertaken in \cite{to} and in this section, we would like to complete the results already obtained.

\subsection{Some former results}\label{rappelsinvariance}
Before restating the main theorem of \cite{to} and other useful consequences, let us recall what are the basic objects involved (see \cite{to}, INTRODUCTION for the details). 

\noindent  Let $T=\frac{i}{\pi}\ddbar {\varphi}$ be a positive $(1,1)$ current representing $c_1(\con)$. One inherits from $T$ a canonically defined (up to a multiplicative constant) $(1,1)$ form $\eta_T$ with $L_{\mbox{loc}}^ \infty$ coefficients wich can be locally written as 

$$\eta_T=\frac{i}{\pi}e^{2\varphi}\omega\wedge\overline{\omega}$$
\noindent where $\omega$ is an holomorphic one form wich defines $\F$ locally.
It turns out that 
$$d\eta_T=0$$

\noindent in the sense of currents (lemme 1.6 in \cite{to}). We can then deduce that $T$ and $\eta_T$ are both {\bf positive $\F$-invariant currents}, in particular the hypersurface $\mbox{Supp}\ N$ (or equivalently the integration current $[N]$) is invariant by $\F$. Moreover, using that $\eta_T$  represents a nef class, one can show that the positive part $Z$ is a non negative multiple of $\{\eta_T\}$ (Proposition 2.14 of \cite{to}). Then, in case that $Z\not=0$, one can normalize $\eta_T$ such that 
$\{\eta_T\}=Z$.

Here is the main result of \cite{to} (Th\'eor\`eme 1 p.368)

\begin{thm}\label{transmetric}
Let $\{N\}+Z$ be the Zariski decomposition of the pseudoeffective class  $c_1(\con)$, then 
there exists a unique $\F$-invariant positive $(1,1)$ current $T$ with minimal singularities (in the sense of \cite{bo})  
such that
\begin{enumerate}

\item $\{T\}=c_1(\con)$ 
\item $T=[N]$ if $Z=0$ \bf{(euclidean type)}.
\item $T=[N]+\eta_T$ if $Z\not=0$ (\bf{hyperbolic type})
\end{enumerate}

\end{thm} 






Actually, the existence of such a current provides a transverse invariant metric for $\F$ (with mild degeneracies on $\mbox{Supp}\ N\cup\mbox{Sing}\ \F$) , namely $\eta_T$,  whose curvature current is $-T$. This justify the words "euclidean" and "hyperbolic".

To each of these transversely invariant hermitian structure, one can associate the {\bf sheaf of distinguished first integrals} ${\mathcal I}^\varepsilon\ \varepsilon=0,1$ depending on whether this structure is euclidean or hyperbolic. 

It will be also useful to consider the sheaf ${\mathcal I }_{d\log}^\varepsilon$ derived from ${\mathcal I}^\varepsilon$ by taking logarithmic differentials. 

These two locally constant sheaves have been introduced in \cite{to} (D\'efinition 5.2) respectively as "faisceau des int\'egrales premi\`eres admissibles" and "faisceau des d\'eriv\'ees logarithmiques admissibles".  Then, we will not enter into further details.  Strictly speaking, they are only defined on $X\setminus \mbox{Supp}\ N$ .

As a transversely homogeneous foliation, $\F$ admits a developing map $\rho$ defined on any covering $\pi:X_0\rightarrow X\setminus \mbox{Supp}\ N$  where $\pi^*\Iadm$ becomes a constant sheaf (\cite{to}, section 6, and references therein).

Recall that $\rho$ is just a section of $\pi^*\Iadm$ (and in particular a holomorphic  first integral of the pull-back foliation $\pi^*\F$) .     
   \noindent One also recovers a representation
   
   $$r:\pi_1(X\setminus \mbox{Supp}\ N)\rightarrow \Im^\varepsilon$$   
      uniquely defined up to conjugation associated to the locally constant sheaf $\Iadm$ and taking values in the isometry group $\Im^\varepsilon$ of $U^\varepsilon$, $U^0=\C,\ U^1=\D$.

\noindent  One can now restate theorem 3 p.385 of \cite{to}.

\begin{thm}\label{complete}
The developing map is complete. That is, $\rho$ is surjective onto $\C$ (euclidean case) or $\mathbb D$ (hyperbolic case).

\end{thm}

One can naturally ask if the fibers or $\rho$ are connected, a question raised in \cite{to}. A positive answer would give useful information on the dynamic of $\F$.  This problem is settled in the next section.

\subsection{Connectedness of the fibers}\label{fibers connectedness}

Let  $\pi:X_0\rightarrow X\setminus\mbox{Supp}\ N$ a  covering  such that the developing map $\rho$ is defined on $X_0$.

Let $K$ be an invariant hypersurface of $\F$: In sufficiently small neighborhood of $K\cup\mbox{Sing}\ \F $, we get a well defined logarithmic $1$ form $\eta_K$  defining the foliation and which is a (semi) local section of $\ladm$ and such that $K\cup\mbox{Sing}\ \F$ is contained in the polar locus of $\eta_K$ (Proposition 5.1 of \cite{to}).  The main properties of $\eta_K$ are listed below and are directly borrowed from ${\it loc.cit}$. They will be used repeatedly and implicitely in the sequel.

\begin{itemize}
\item The residues of $\eta_K$ are non negative real numbers.  To be more precise, the gerrm of $\eta_K$ at  $x\in K\cup \mbox{Sing}\ F$ can be written as 

  $$\eta_K=\sum_{i=1}^{n_x}\lambda_i\frac{df_i}{f_i}.$$

 \noindent  where the  $f_i\in {\mathcal O}_{X,x}$ are pairwise irreducible and the $\lambda_i$'s are $\geq 1$.
 
 In particular, the multivaluate section  $e^{\eta_K}=\prod {f_i}^{\lambda_i}$ of $\Iadm$ is an elementary first integral of $\F$. 
 
  \item If $f_i=0$ is a local equation of a local branch $N_{f_i}$ of $\mbox{Supp}\ N$, then $\lambda_i= \lambda_{f_i} (N_{f_i}) +1$ where   $\lambda_{f_i} (N_{f_i})$ stands for the weight of $N_{f_i}$ in $N$ and $\lambda_i=1$ otherwise.
Thus,  $e^{\eta_K}$ is univaluate and reduced (as a germ of ${\mathcal O}_{X,x}$) if and only if $x\notin \mbox{Supp}\ N$.
\item The local separatrices at $x$ (i.e the local irreducible invariant hypersurfaces containing $x$) are exactly those given by $f_i=0$.

\item $x$ is a regular point of $\F$ if and only if $n_x=1$ and $f_1$ is submersive at $x$.

\end{itemize}
\vskip 5 pt

Our aim is to show that the fibers of $\rho$ are connected if one allows to identify certain components contained in the same level of $\rho$. On the manifold $X$, this corresponds to  identify some special leaves in a natural way. To clarify this, one may typically think to the fibration defined by the levels of $f=xy$ on ${\C}^2$. All the fibers are connected. However, if we look at the underlying foliation, the fiber $f^{-1}(0)$ consists in the union of two leaves that must be identified if one want to make the space of leaves Hausdorff. Observe also that, when blowing up the origin, these two special leaves intersect the exceptional divisor $E$ on the new ambient manifold $\tilde{{\C}^2}$  . Coming back to our situation, one may think that leaves wich ``intersect'' the same components of $\mbox{Supp}\ N$ must be identified. This indeed occurs if we replace the general $X_0$ by the basic previous example $\tilde{{\C}^2}\setminus E$.
\vskip 5 pt
More precisely, we proceed as follows:
For $p\in X_0$, denote by ${\mathcal C}_p$ the connected component of $\rho^{-1}(\rho (p))$ containing $p$.
Let $p,q\in X_0$. Define the binary relation $\mathcal R$ by $p\mathcal R q$ whenever 

\begin{enumerate}
 \item ${\mathcal C}_p={\mathcal C}_q$

or

\item $\rho (p)=\rho (q)$ and there exists a connected component $H$ of $\mbox{Supp}\ N$ such that on any sufficiently small neighborhood of $H$,  ${(\eta_H)}_\infty\cap\pi ({\mathcal C}_p)\not=\emptyset$ and ${(\eta_H)}_\infty\cap \pi({\mathcal C}_q)\not=\emptyset$.
\end{enumerate}

We will denote by $\overline{\mathcal R}$ the equivalence relation generated by $\mathcal R$ and $A_{\overline{\mathcal R}}$ the saturation of $A\subset X_0$ by $\overline{\mathcal R}$.

\begin{remark}\label{finitelymanyleaves}

The saturation ${({\mathcal C}_p)}_{\overline{\mathcal R}}$ of a connected component of a fiber of $\rho$ is generically reduced to itself. In fact, as a consequence of the above description of $\eta_K$,  this holds outside a countable set of components wich projects by $\pi$ onto a finite set of leaves of $\mathcal F$ whose cardinal is less or equal to the number of irreducible components of ${(\eta_K)}_\infty $ with $K=\mbox{Supp}\ N$. 

\end{remark}

\vskip 20 pt

Consider now the quotient space $X_0/\overline{\mathcal R}$ endowed with the quotient topology; as an obvious consequence of our construction, $\rho$ factorize throught a surjective continuous map $\overline{\rho}:X_0/\overline{\mathcal R}\rightarrow U^{\varepsilon}$.  

Using that the singularities of $\F$ are elementary and \cite{to}( Lemme 6.1, corollaire 6.2), we get the 
\begin{lemma}\label{satopen}
 The saturation ${V}_{\overline{\mathcal R}}$ of an open set $V$ is  open in $X_0$.
\end{lemma}

\begin{remark}\label{deckequiv}
 When $X_0$ is Galoisian, any deck transformation $\tau$ of $X_0$, $r$ is compatible with the equivalence relation above there exists a unique set theoretically automorphism $\overline \tau$ of $X_0/\overline{\mathcal R}$ such that  $\rho\circ \tau={\overline \rho}\circ\overline{\tau}\circ\varphi_{\overline{\mathcal R}}$ where $\varphi_{\overline{\mathcal R}}:X_0\rightarrow X_0/\overline{\mathcal R}$ denotes the canonical projection. Observe also that $\overline{\tau}$ is an homeomorphism, by definition of the quotient topology.
\end{remark}

The following result somewhat teaches us that, modulo these identification, the fibers of $\rho$ are connected.

\begin{thm}\label{connect}
 The map $\overline{\rho}:X_0/\overline{\mathcal R}\rightarrow U^\varepsilon$ is one to one.

\end{thm}

As a straightforward consequence, we obtain the

\begin{cor}\label{Ntrivial}
 Suppose that the negative part $N$ is trivial, then the fibers of $\rho$ are connected.
\end{cor}

The proof is in some sense close to that of the completeness of $\rho$ with makes use of a metric argument (see \cite{to}). In this respect, it will be useful to recall the following definition  (see \cite {mo}).
\begin{definition} 
Let $\F$ a regular transversely hermitian codimension $1$  holomorphic foliation on a complex manifold $X$. Let $h$ be the metric induced on $N_\F$ by the invariant  transverse hermitian metric. Let $g$ be an hermitian metric  on $X$ and $g_N$ the metric induced on $N_\F$ by $g$ via the canonical identification ${T_\F}^{\perp}\simeq N_\F$.  

The metric $g$ is said to be {\bf bundle-like} ("quasi-fibr\'ee" in french terminology) if $g_N=h$.
\end{definition}

\begin{remark}
Because of the "codimension 1" assumption on $\F$, one can  notice that, when $\F$ is transversely hermitian, any hermitian metric on $X$ is conformally equivalent to a bundle-like one .
\end{remark}

The following property is used in \cite{to} (Lemme 6.3):
\begin{prop} (cf \cite{mo} , Proposition 3.5 p.86).
Let $\F$ a transversely hermitian foliation equipped with a bundle like metric $g$. 
Let $\gamma:I\rightarrow X$ a geodesic arc (with respect to $g$). Assume that $\gamma$ is orthogonal to the leaf passing through  $\gamma (t_0)$  for some $t_0\in I$. Then, for every $t\in I$,  $\gamma$ remains orthogonal  to the leaf passing through $\gamma (t)$.

\end{prop}

\begin{remark}
 In our setting, it is also worth recalling (\it loc.cit) that $\rho(\gamma)$ is a geodesic arc of $U^\varepsilon$ endowed with its natural metric $ds^\varepsilon$ depending on the metric type of $N_\F$ (euclidean or hyperbolic).
\end{remark}

\noindent {\it Proof of} theorem \ref{connect}

 It is the consequence of  the following sequence of observations.  Consider firstly a bundle-like metric $g$ defined on the complement of $E=\mbox{Supp}\ N\cup \mbox{Sing}\ \mathcal F$ in $X.$  For $p\in X\setminus E$ and $\varepsilon >0$, $B(p,\varepsilon)$ stands for the  closed geodesic ball of radius $\varepsilon$ centered at $p$. This latter is well defined is $\varepsilon$ is small enought (depending on $p$).

Acording to \cite{to}, prop 5.1, there exists a logarithmic $1$ form $\eta$ with positive residues defining the foliation near $E=\mbox{Supp}\ N\cup \mbox{Sing}\ \mathcal F$ (keep in mind that $\eta$ is just a semi-local section of the sheaf $\ladm$ mentioned above).
Let $C_1,...,C_p$ the connected component of $E$. Pick one of these, says $C_1$. Following \cite{pa2}, there exists an arbitrary small connected open neighborood $U_1$ of $C_1$ such that the fibers of $f_1=\exp \int\eta$ (defined as a multivaluate fonction) on $U_1\setminus {(\eta)}_\infty$ are connected and one can also assume that the polar locus $\eta_\infty$ is connected. This latter can be thought as the ``singular fiber'' of $f_1$. Actually, assumptions of theorem B of  \cite{pa2} are fullfilled: this is due to the fact that the residues of $\eta$  are positive  equal to $1$ on the local branches of $(\eta)_\infty$ intersecting 
$C_1$ 
   not contained in $C_1$, noticing moreover that this set $S^1$ of local branches is {\bf non empty} (see \cite{to}, Proposition 5.1 and 5.2), a fundamental fact for the following. Let $D_1\subset C_1$ the union of irreducible codimension $1$ components; in other words, $D_1$ is the connected component of $\mbox{Supp}\ N$ contained in $C_1$.  Remark that $S_1=f_1^{-1}(0)$ when extending $f_1$ throught $S_1\setminus D_1$. Observe also that $S_1$ is not necessarily connected, as the ``negative'' component $D_1$ has been removed. This justify the definition of the previous equivalence relation $\overline{\mathcal R}$.

Let $W_1\Subset V_1\Subset U_1$ some open neighborood of $C_1$. One can perform the same on each component and we thus obtain open set $W_i\Subset V_i\Subset U_i\ i=1...p$ with $f_i, S_i$ similarly defined.

Let $U=\bigcup_iU_i,\ V=\bigcup_iV_i,\ W=\bigcup_iW_i$. Up to shrinking $U_i$, $V_i$ and $W_i$, one can assume that the following holds:

\begin{enumerate}
 \item $U_i\cap U_j=\emptyset,\ i\not= j$
\item There exists $\alpha>0$ such that for every $m\in X\setminus \overline{W}$, $B(m,\alpha)$ is well defined.
\item \label{ballcentered}
Let $x\in V_i\setminus \mbox{Supp}\ N$, and $f_{i,m}$ a determination of $f_i$ at $m$, then the image of $U_i\setminus \mbox{Supp}\ N$ by $f_i$ contains the closed geodesic ball (euclidean or hyperbolic)  of radius $\alpha$ centered at $f_{i,m} (m)$, with $f_{i,m}$ any germ of determination of $f_i$ in $m$.

\end{enumerate}

 

Indeed, one can exhibit such $V_i$ (the only non trivial point) in the following way. Pick $x\in S_i$ such that $\F$ is smooth on a neighborhood of $x$ and take a small holomorphic curve $$\gamma:\{|z|<\varepsilon\}\rightarrow U_i,\ \ \gamma(0)=x$$ 

\noindent tranverse to $\F$ such that property (3) holds for every $y\in\mbox{Im}\ \gamma$ (this means that $\alpha$ can be chosen independently of $y$). This is possible because $f_i$ (more exactly any of its determinations) is submersive near $x$. By \cite{to} (Corollaire 6.2), the saturation of $\mathcal C=\mbox{Im}(\gamma)$ by ${\F}_{|U_i}$ is also an open neighborhood of $C_i$ and one can take $V_i$ equal to it. Note also, that the property (\ref{ballcentered}) mentioned above do not depend on the choice of the determination, as others are obtained by left composition through a global isometry of $U^\varepsilon$.
\begin{lemma}
 The quotient space $X_0/\overline{\mathcal R}$ is Hausdorff (with respect to the quotient topology).
\end{lemma}

\begin{proof} 
Suppose that $X_0/\overline{\mathcal R}$ is not Hausdorff. Using lemma \ref{satopen} and remark \ref{finitelymanyleaves}, one can then find two points $p,q$ belonging to the same fiber  of $\rho$ such that $\{p\}_{\overline{\mathcal R}}\not=\{q\}_{\overline{\mathcal R}}$ and a sequence $(p_n),\ p_n\in X_0$ converging to $p$ such that 
\begin{enumerate}
 \item ${\{p_n\}}_{\overline{\mathcal R}}$ coincide with the connected component of $\rho^{-1}(\rho (p_n))$ containing $p_n$ and such that ${\{p_n\}}_{\overline{\mathcal R}}$ does not intersect the singular locus of the pull-back foliation $\pi^*\mathcal F$ (in other words, $\rho$ is submersive near every point of ${\{p_n\}}_{\overline{\mathcal R}}$).

\item  For every neighborood $\mathcal V$ of $q$, there exists $n$ such that ${\{p_n\}}_{\overline{\mathcal R}}\cap{\mathcal V}\not=\emptyset$. Equivalently, up passing to a subsequence,  one can find  $(q_n),\  q_n\in{\{p_n\}}_{\overline{\mathcal R}}$ converging to $q$. In particular, $\lim\limits_{n \to +\infty}\rho (q_n)=\rho (q)=\rho (p)$.
\end{enumerate}
Without any loss of generality, we may also assume that $\rho$ is submersive near $p$ and $q$.

Let $\gamma_n:[0,1]\rightarrow {\{p_n\}}_{\overline{\mathcal R}}$ be a continuous path joining $p_n$ to $q_n$. 

Let $t\in[0,1]$ and $p_n(t)=\rho (\gamma_n(t))$. In $[0,1]$, consider the open sets $O_n^1=p_n^{-1}(U\setminus \overline{W})$ and $O_n^2=p_n^{-1}(V)$.

One can choose $n$ big enought such that the following constructions make sense.

If $p_n(t)\notin \overline{W}$, there exists a unique geodesic arc in $B(p_n(t),\alpha)$ orthogonal to the leaves wich lifts in $X_0$ to a path joining $\gamma_n(t)$ to a point $r_n(t)$ such that $\rho (r_n(t)=\rho (p)$. 

In that way, one can define a map  $$\alpha_n:O_n^1\rightarrow X_0/\overline{\mathcal R}$$
setting $\alpha_n(t)= {\{r_n(t)\}}_ {\overline{\mathcal R}}.$ Note that $\alpha_n$ is {\it locally constant}.

Now suppose  that  $p_n(t)\in V_i$ for some $i$ and denote by $\pi^{-1}$ the local inverse of the covering map $\pi$ sending $p_n(t)$ to $\gamma_n(t)$. 
 Keeping track of the conditions imposed on $V_i\Subset U_i$, there exists a path in $U_i\setminus \mbox{Supp}\ N$ starting from $p_n(t)$  such the analytic continuation of   
 $\rho\circ {\pi}^{-1}$  has $\rho (p)$ as ending value. This path lifts in $X_0$ to a path joining $\gamma_n(t)$ to a point  $s_n(t)$.

\begin{remark}\label{path}Note that the point $s_n(t)$ depends on the choice of this path. However, 
as a consequence of the connectedness of the fibers of $f_i$, the equivalent class ${\{s_n(t)\}}_ {\overline{\mathcal R}}$ {\it is uniquely defined}.
\end{remark}

This allows us to define a map 

$$\beta_n:O_2^n\rightarrow X_0/\overline{\mathcal R}$$

\noindent setting $\beta_n(t)={\{s_n(t)\}}_ {\overline{\mathcal R}}$ which is also obviously locally constant. 

\noindent Moreover, by remark \ref{path}, $\alpha_n$ and $\beta_n$ coincide on $O_1^n\cap O_2^n$.

We have thus constructed a {\it locally constant} map from $[0,1]$  to  $X_0/\overline{\mathcal R}$ sending $0$ to $\{p\}_{\overline{\mathcal R}}$ and $1$ to $ {\{q}\}_{\overline{\mathcal R}}$. This provides the sought contradiction.

\end{proof}
{\it Proof of the theorem \ref{connect}.}

By lemma \ref{satopen}, the map $\overline{\rho}$ is a local homeomorphism. Moreover, $X_0/\overline{\mathcal R}$ is Hausdorff. Hence,  $X_0/\overline{\mathcal R}$ carries, via $\overline{\rho}$, a uniquely defined connected Riemann surface structure wich makes $\overline{\rho}$ analytic.

Moreover the pull-back by $\overline{\rho}$ of the metric $ds^\varepsilon$ is complete (this is straighforward consequence of the proof of theorem \ref{complete}). Hence, it turns out that $\overline{\rho}$ is indeed an analytic covering map. The space $U^\varepsilon$ being simply connected, one can conclude that $\overline{\rho}$ is actually a biholomorphism.\qed.

This allows us to recover (modulo further informations postponed in the sequel (see section \ref{kodaira}) a result due to Carlos Simpson who has proved in \cite{sim1} (among other things)

\begin{thm}
 Let $\omega$ an holomorphic one form on $X$ K\"ahler compact and let $\tau:\tilde X\rightarrow X$ a covering map such that

     $$g:\tilde X\rightarrow\C$$
obtained by integration of $\omega$ is well defined. Then we get the following alternative:

\begin{enumerate}
 \item $\omega$ factors throught a surjective morphism $X\rightarrow B$ on a compact Riemann surface $B$,
\item the fibers of $g$ are connected.
\end{enumerate}

\end{thm}

\begin{proof}
 Let $\kappa (N_\mathcal F^*)$ be the Kodaira dimension of the conormal bundle of the foliation $\mathcal F$ defined by $\omega$.

If $\kappa((N_\mathcal F^*))=1$, one can apply the Castelnuovo-De Franchis theorem on a suitable ramified cyclic covering and show that $\mathcal F$ define a fibration over an algebraic curve $B$(see \cite{Reid}); moreover ({\it loc.cit}, see also \cite{brme}), this dimension can't exceed one. Then one can assume that $\kappa((N_\mathcal F^*))=0$; this is equivalent to say that $c_1(N_\F^*)$ has zero numerical dimension (see theorem \ref{abundancecases}, section \ref{kodaira}) and this implies that $U^\varepsilon=U^0=\C$. Let $\mathcal I$ the sheaf on $X$ whose local sections are primitives of $\omega$. In restriction to the complement of $\mbox {Supp}\ N$ we thus have $\mathcal I={\mathcal I}^0$ 

Keeping the previous notations, one can take $X^0$ equal to ${\tilde X}\setminus g^{-1}(H)$, setting $H=\mbox{Supp N})$.
In particular, one can conclude when $N$ is trivial (corollary \ref{Ntrivial} ). 


We will denote by $g^0$ the restriction of $g$ to $X^0$. This is nothing but the developing map considered previously, in particular,  the image of $g^0$ is $\C$. Let $A=\{a\in\C\ \mbox{such that} \ g^{-1}(a)\cap \tau^{-1}(H)\not=\emptyset\}$. Following theorem \ref{connect}, the fiber $g^{-1}(x)$ is connected whenever $x\notin A$. We want to show that the same holds true if $x\in A$. Actually this easily results from the definition of $\overline{\mathcal R}$, theorem \ref{connect}  and the following fact: for every point $y\in H$, there exists  a continuous path $\gamma:[0,1]\rightarrow H$ such that $\gamma (0)=y$ and $\gamma(1)$ belongs to a local separatrix wich is not a local branch of $H$ ( \cite{to}, Proposition 5.1 and 5.2).

\end{proof}

\section{Uniformization and dynamics}\label{sectiondynamics}

 We keep the same assumptions and notations of the previous section, in particular, unless otherwise stated,  $X$ still denotes compact K\"ahler manifold carrying a codimension $1$ foliation such that $N_\mathcal F^*$ is pseudoeffective. 
 
 The Hermitian transverse structure of the foliation $\mathcal F$,  as explicited in \cite{to} and recalled in section \ref{psefcon}  , may degenerate on the negative part with very ``mild'' singularities, as described in {\it loc.cit}.

 Hence, one can guess that this kind of degeneracy will not deeply affect the dynamics of leaves in comparison to the usual case: transversely riemannian foliations on (real) compact manifolds.
   Let us recall some properties  in this setting (see \cite{mo}).
\vskip 5 pt
-The manifold is a disjoint union of minimals (recall that a minimal is a closed subset of $X$ saturated by the foiation and minimal with respect to this properties). 
\vskip 5 pt
-When the real codimension $\mathcal F$ is equal to $2$, a minimal $\mathcal M$ fall into one of the following three types:

\begin{enumerate}
 \item $\mathcal M$ is a compact leaf.
\item $\mathcal M$ coincide with the whole manifold.
\item $\mathcal M$ is a real hypersurface.
\end{enumerate}

Now, let's come back to our situation.

Recall that  $\Im^\varepsilon$ denotes the isometry group of $U^\varepsilon$, $\varepsilon\in\{0,1\}$ (keep in mind that  $U^0=\C$ and $U^1=\mathbb D$). The locally constant sheaf $\Iadm$ induces a representation 

      $$r: \pi_1(X\setminus \mbox{Supp}\ N)\rightarrow \Im^\varepsilon$$
Let $G$ the image of the representation $r$ and $\overline{G}$ its closure (with respect to the usual topology) in $\Im^\varepsilon$.

Let $\rho:X_0\rightarrow U^\varepsilon$ be the associated developping map.
For the sake of convenience, we will assume that the covering $\pi:X_0\rightarrow X\setminus \mbox{Supp}\ N$ is Galoisian and as such satisfies the following equivariance property

    $$\rho\circ \gamma:r (\gamma)\circ\rho$$
    
    \noindent where $\gamma$ in the right side is an element of $\pi_1(X\setminus\mbox{Supp}\ N)$, in the left one is the corresponding deck transformation. 
    

\begin{definition} The group $G$ is called the monodromy group of the foliation $\mathcal F$.
 
\end{definition}
\vskip 10 pt

We would like to recover the aforementioned  dynamical properties 
 for regular foliations. As we deal with singular foliations, we will have to modify slightly the usual definition of  a leaf, taking into account the identifications made in a previous section.

To this end, let $c\in U^\varepsilon$ and recall that $\rho^{-1}(c) ={(\rho^{-1} (c))}_{\overline{\mathcal R}}$, that is coincide with its saturation by the  equivalence relation $\overline {\mathcal R}$ (theorem \ref{connect}). Let $F_c=\pi (\rho^{-1} (c))$; as a consequence of the previous construction, we have $F_c=F_{c'}$ if and only if there exists $g\in G$ such that $c'=g(c)$ (see remark \ref{deckequiv}). Let ${\mathcal C}_1,....,{\mathcal C}_p$ the connected components of $\mbox{Supp}\ N$ such that ${(\eta_{\mathcal C_i})}_\infty\cap F_c\not=\emptyset$ and let ${\mathcal L}_c=F_c\cup_i{\mathcal C}_i$. Of course, this family of components depends  on the choice of $c$.

\begin{definition}\label{modifiedleaf}
 A subset of $X$ of the form ${\mathcal L}_c$ is called {\bf modified leaf} of the foliation $\mathcal F$.
\end{definition}

\begin{definition}
 A subset $\mathcal M$ is called {\bf modified minimal} of the foliation whenever $\mathcal M$ is the closure of a modified leaf and is minimal for the inclusion.
\end{definition}

By remark \ref{finitelymanyleaves}, all modified leaves, except a finite number, are indeed ordinary leaves. 

The family $\{{\mathcal L}_c\}$ defines a partition of $X$. Let $X/\mathcal F$ the associated quotient space which, in some sense represents the space of leaves. One obtain a canonical bijection $$\Psi:X/\mathcal F\rightarrow U^\varepsilon/G$$ induced by  ${\mathcal L}_c\rightarrow c$. 

One can translate the previous observations into the following statement.

\begin{thm}
 The map $\Psi$ is a homeomorphism w.r.t the quotient topology on each space
\end{thm}

By compactness of $X$, we obtain the 
\begin{cor}
 The group $G$ is cocompact, i.e: $U^\varepsilon/G$ is a (non necessarily Hausdorff) compact space for the quotient topology
\end{cor}

With this suitable definition of the leaves space $X/\mathcal F$, one can observe that the dynamic behaviour of $\mathcal F$ is faithfully reflected in those of the group $G$. The next step is then to describe its topological closure $\overline{G}$.

Assume firstly that $U^\varepsilon=\mathbb D$ and $G$ is affine (i.e: $G$ fix a point in $\partial \mathbb D $). We want to show that this situation can't happen. 

In this case, $G$ does not contain any non trivial elliptic element. In particular, this means that $F=e^{\int\eta_H}$ is a well defined (univaluated) {\it holomorphic} first integral of the foliation on a neighborhood $V$ of $H=\mbox{supp}\ N$. Futhermore, one can notice that the zero divisor of the differential $dF$ coincide with the negative part $N$. Using again that $G$ is affine, one eventually obtains that $\mathcal F$ is defined by a section of $N_\F^*\otimes{\mathcal O}(-N)\otimes E$ without zeroes in codimension $1$ where $E$ is a flat line bundle (take the half plane model instead of $\mathbb D$).This contradicts the fact that the positive part $Z$ is non trivial.

Once this case has been eliminated, the remaining ones can be easily described as follows (taking account that $G$ is cocompact):

\begin{prop}\label{StructureG}
 Up to conjugation in ${\Im}^\varepsilon$, the toplological closure $\overline{G}$ has one of the following form:

\begin{enumerate}
 \item \label{lattice} $G=\overline{G}$ and then is a cocompact lattice.
\item\label{minimal} The action of $G$ is minimal and
\begin{enumerate}
 \item either $\varepsilon =0$ and $\overline G$ contains the translation subgroup of ${\Im}^0$,
\item either $\varepsilon=1$ and $\overline {G}= \Im^1$.
\end{enumerate}
\item \label{hypersurface} $\varepsilon=0$ and $\overline{G}$ contains the subgroup of real translation , $T=\{t_\alpha, \alpha\in\R\} (t_\alpha (z)=z+\alpha)$ and more precisely 
\begin {enumerate}
\item either $\overline{G}=\langle T,t\rangle$,
\item  either 
$\overline{G}=\langle T,t,s\rangle$
\end{enumerate}
with $t(z)=z+ai$ for some real $a\not=0$  and $s(z)=-z$.

\end{enumerate}
\end{prop}
When considering the foliation viewpoint, this latter description leads to the 

\begin{THM}\label{leafdynamic}
 The topological closure of every modified leaf is a modified minimal, the collection of those modified minimals form a partition of $X$ , moreover 
\begin{enumerate}
 \item in case (\ref{lattice}) of prop. \ref{StructureG}, Modified leaves coincide with their closure and $\mathcal F$ defines, via $\Psi$, an holomorphic fibration over the curve $U^{\varepsilon}/G$;
\item in case (\ref{minimal}) of prop. \ref{StructureG}, $X$ is the unique modified minimal;
\item in case (\ref{hypersurface}) of prop. \ref{StructureG}, every modified minimal is a real analytic hypersurface.

\end{enumerate}

\end{THM}

\begin{remark}\label{closedallclosed}
 It is worth noticing that the existence of a closed modified leaf implies that the other ones are automatically closed (case 1 of the previous theorem).
\end{remark}

\begin{cor}\label{boundedfamily}Assume that $\mathcal F$ is not a fibration, then the family of irreducible hypersurfaces invariant by $\mathcal F$ is exceptional and thus has cardinal bounded by the picard number $\rho (X)$.
 \end{cor}

\begin{proof}
Assume that there exists a non exceptional family of invariant hypersurfaces. One can then find a connected hypersurface $K=K_1\cup K_2\cup ... \cup K_r$ invariant by $\mathcal F$ such that the family of irreducible components $\{K_1,...,K_r\}$ is not an exceptional one. By virtue of Proposition 5.2 of \cite{to} , $K$ is necessarily a modified leaf and $\mathcal F$ is tangent to a fibration by remark \ref{closedallclosed}.
\end{proof}

\section{Kodaira dimension of the conormal bundle}\label{kodaira}

As previously, $\F$ is a codimenson $1$ holomorphic  foliation on a compact K\" abler manifold $X$ such that $N_\FÃÅ¡^*$ is pseudoeffective. We adopt the same notations as before.

In \cite{to} (see Remarque 2.16 in {\it loc.cit}), it has been proved that for $\alpha=c_1(N_\F^*)$, the numerical dimension $\nu(\alpha)$ takes value in $\{0,1\}$ depending of the metric type of $N_\F$; $0$ corresponds to the Euclidean one and $1$ to the hyperbolic one (see theorem \ref{transmetric}).  Moreover, we always have 

$$\kappa (N_\F^*)\leq \nu(\alpha)$$
where the left member represents the Kodaira dimension of $N_\F^*$. This is a common feature for line bundles (see \cite{na}) that can be directly verified in our case. Indeed, we always have $\kappa  (N_\F^*)\leq 1$ (cf. \cite{Reid}) and when $\alpha=\{N\}$ (in particular, $N=\sum_i\lambda_i D_i$ is a $\Q$ effective divisor), $[N]$ is the only positive current wich represents $\{N\}$. This latter property implies that  $\kappa(N_\F^*)\leq\kappa (N)=0$.  

Among the different cases described in prop.\label{structureG} theorem \ref{leafdynamic}, we would like to characterize those for wich abundance holds, that is $\kappa (N_\F^*)= \nu(\alpha)$.

\begin{THM}\label{abundancecases}
 Let $\alpha=c_1(N_\mathcal F^*)$. Then, $\nu (\alpha)=\kappa (N_\F^*)$ in both cases:
\begin{enumerate}
 \item $\varepsilon=0$ (Eulidean type)\label{parabolic}.
\item \label{latticecase}
 $\varepsilon=1$ (Hyperbolic type) and $G$ is a lattice.

\end{enumerate}

Moreover, in the remaining case $\varepsilon=1$ and $G$ dense in ${\Im}^1$, we have $\kappa (N_\F^*)=-\infty$ and $\nu(\alpha)=1$.

\end{THM}

\begin{proof}
 
Let us firstly settle the case $\varepsilon=1$ and $\overline{G}={\Im}^1$

Assume that $\kappa(N_\F^*)\geq 0$. 
Consider in first place the simplest situation $h^0(N_\mathcal F^*)\not=0$; this means that the foliation is defined by an holomorphic $1$ form $\omega$. Taking local primitives of $\omega$, we can equipped the foliation with a transverse projective structure  different from the hyperbolic one (attached to $\varepsilon=1$). By means of the Schwarzian derivative, one can thus construct, as usual, a quadratic differential of the form $$F\omega\otimes\omega$$

\noindent where $F$ is a meromorphic first integral of $\mathcal F$. For obvious dynamical reasons, $F$ is automatically constant and this forces local first integrals $g\in {\mathcal I}^1$ to be expressed in the form 
$g=e^{\lambda f}$ (modulo the left action of $\Im^1$). In particular the monodromy group $G$ must be abelian and this leads to a contradiction.

In the more general setting, the same conclusion holds with the same arguments, considering a suitable ramified covering $\pi:X_1\rightarrow X$ such that the pull-back foliation is defined by an holomorphic form.

\bigskip

Concerning the case \ref{latticecase}) of theorem \ref{abundancecases} , we know that $\mathcal F$ is a fibration over the compact riemann surface $S=\mathbb D/G$. For $n\gg 0$ one can then find two independant holomorphic sections $s_1,s_2$ of ${K_S}^{\otimes n}$ (with respect to the hyperbolic orbifold structure).  One must take care that the map $\Psi$ is not necessarily  a morphism in the orbifold sense.  Indeed, the local expression of $\Psi$  along $\mbox{Supp}\  N$ is given, up to left composition by an element of $\Im^\varepsilon$,  as a multivaluate section of $\Iadm$ of the form $f_1^{\lambda_1}....f_r^ {\lambda_r}$ (see remark \ref{finitelymanyleaves}). Hence, it fails to be an orbifold map through $\mbox{Supp}\ N$ unless the $\lambda_i $ are integers.  However, the fact that they are $> 1$ (as pointed out in {\it loc.cit}) guarantees that $s_1$ and $s_2$ lift (via $\Psi$) to $2$ independant {\it holomorphic} sections  of ${N_\mathcal F^*}^{\otimes n}$, whence $\kappa (N_\mathcal F^*)=1$.
\bigskip

The remaining case $(1)$ is a bit more delicate to handle. Here, the Kodaira dimension $\kappa (N_{\F}^*)$ is directly related to the linear part $G_L$. Namely, $\kappa (N_{\F}^*)=0$ if and only if $G_L$ is finite. This occurs in particular when $\F$ is a fibration. 

One can then eliminate this case and suppose from now on that the generic leaf of $\F$ is not compact.  Assume firstly that $N$ is an effective integral divisor (a priori, it is only $\Q$ effective). This means that the foliation is defined by a twisted one form $\omega\in H^0(X,\Omega^1\otimes E)$ with value in a numerically trivial line bundle $E$, that is $E\in{\mbox{Pic}}^\tau (M):=\{L\in\mbox{Pic}\ (M)|c_1(L)=0\in H^2(M,\R)\}$. Note that the zeroes divisor of $\omega$ is precisely $N$ and the theorem will be established once we will prove that $E$ is actually a torsion line bundle. 

To this aim, remark that the dual $E^*$ is an element of the so-called Green-Lazarsfeld subset  

$$S^1(X)=\{L\in {\mbox{Pic}}^\tau\ \mbox{such that}\ H^1(X,L)\not=0\}.$$

In the sequel, it will be useful to consider $H^1(X,{\C}^*)$ as the paramerizing space of rank one local systems and to introduce the set 

   $$\Sigma^1(X)=\{\mathcal L\in H^1(X,{\C}^*)\ \mbox{such that}\ H^1(X,\mathcal L)\not=0\}$$

When $\mathcal L$ is unitary, $S^1(X)$ and $\Sigma^1(X)$ are related as follows:

$$\mathcal L \in \Sigma^1(X)\ \mbox{if and only if}\ L\in S^1(X)\cup-S^1(X)$$

where $L=\mathcal L\otimes {\mathcal O}_X$ is the numerically trivial line bundle defined by $\mathcal L$ (see \cite{beau}, a) Proposition 3.5).

When $X$ is projective, Simpson  has shown that the isolated points of $\Sigma^1(X)$ are torsion characters (\cite{sim}). 
This has been further generalized by Campana in the K\" ahler setting (\cite{camp}).

We proceed by contradiction. Assume that $\kappa(N_\F^*)\not=0$, in other words, $E$ is not torsion. 

This can be translated into the existence of a flat unitary connection 

$$\nabla_u:E\rightarrow \Omega^1(E)$$ whose monodromy representation has infinite image $H$ in the unitary group $U(1)$. As mentioned before, $H$ is nothing but the linear part $G_L$ of the monodromy group $G$ attached to the foliation 

Keeping track of Simpson's result, $\mathcal E=\mbox{ker}\ \nabla_u$ is not isolated in $\Sigma^1(X)$, hence, following \cite{beau}, one can conclude to the existence of an holomorphic fibration with connected fibers over a compact Riemann surface of genus $g\leq 1$ such that ${\mathcal E}\in p^*(H^1(B,{\C }^*)$, up replacing $X$ by a finite \' etale covering. 





\vskip 5 pt
From $\nabla_u \omega=0$ and the fact that $E$ comes from a line bundle on $B$, one can conclude that $\omega$ restricts to an holomorphic closed one form on open sets $p^{-1}(V)$, $V$ simply connected open set in $B$.

Let $F$ be a smooth fiber of $p$ and $\Gamma < \mbox{Aut}(H^1(F,\C))$ the monodromy group associated to the Gauss-Manin connection. Let $D$ the line spanned by $\omega_{|F}$ in $H^1(F,\C)$ (recall that the foliation is not tangent to the fibration).

One can observe than $D$ is globally $\Gamma$ invariante (actually, $H$ is just the induced action of $\Gamma$ on $D$).  By \cite{del}, Corollaire 4.2.8, or the fact that the connected component of ${\overline{\Gamma}}^{{Zar}}$ is semi-simple (see {\it loc.cit}), one can then infer that $H$ is finite.  


\vskip 5 pt

 The fact that $N$ is in general only effective over $\mathbb Q$ does not really give troubles. One can easily reduce this case to the previous one, working on a suitable ramified covering.


\end{proof}
 \begin{remark}
 
 Deligne's results quoted above are stated in an algebraic setting but the semi-simplicty of  ${\overline{\Gamma}}^{{Zar}}$ remains true in the K\" ahler realm, as pointed out in \cite{camp}.
\end{remark}

\section{The non abundant case}\label{nonabundant}
We are dealing with the non abundant case: $\kappa(N_\F^*)=-\infty<\nu(N_\F^*)=1$. From now on, we will also assume that the ambient manifold $X$ is {\bf projective}.

As a consequence of theorem \ref{abundancecases}, we know that $\F$ is a {\bf quasi-minimal} foliation (i.e: all but finely many leaves are dense).

The foliation $\F$ fall into the general setting of transversely projective foliations. For the  reader convenience, we now recall some definitions/properties concerning these latters.

\subsection{Transversely projective foliations, projective triples}\label{projectivetriple}

We follows closely the presentation of \cite{lope} (see also \cite{lptbis}) and we refer to {\it loc.cit} for precise definitions. A transversely projective foliation $\mathcal F$ on a complex manifold $X$ is the data of $(E,\nabla,\sigma)$ 
where 
\begin{itemize}
\item $E\to X$ is a rank $2$ vector bundle, 
\item $\nabla$ a flat meromorphic connection on $E$ and
\item $\sigma:X\rightarrow P=\mathbb P (E)$ a meromorphic section generically transverse to the codimension one Riccati foliation $\mathcal R=\mathbb P(\nabla)$ and such that $\mathcal F=\sigma^*\mathcal R$. 
\end{itemize}

Let us assume (up to birational equivalence of bundles) that $E=X\times {\mathbb P}^1$ is the trivial bundle (this automatically holds when $X$ is projective by GAGA principle), $\sigma$ is the section $\{z=0\}$
so that the Riccati equation writes
$$\Omega=dz+\omega_0+z\omega_1+z^2\omega_2$$
with $\omega_0$ defining $\mathcal F$.
Setting $z=\frac{z_2}{z_1}$ we get the $\mathfrak{sl}_2$-connection (i.e. trace free)
$$\nabla\ :\ Z=\begin{pmatrix}z_1\\ z_2\end{pmatrix}\ \mapsto\ dZ+\begin{pmatrix}\alpha&\beta\\ \gamma&-\alpha\end{pmatrix}Z$$
where
$$ \begin{pmatrix}\alpha&\beta\\ \gamma&-\alpha\end{pmatrix}:=\begin{pmatrix}-\frac{1}{2}\omega_1&-\omega_2\\ \omega_0&\frac{1}{2}\omega_1\end{pmatrix}.$$
Note that 
$$\nabla\cdot\nabla=0\ \Leftrightarrow\ \Omega\wedge d\Omega=0\ \Leftrightarrow\ \left\{\begin{matrix}
d\omega_0=\hfill\omega_0\wedge\omega_1\\
d\omega_1=2\omega_0\wedge\omega_2\\
d\omega_2=\hfill\omega_1\wedge\omega_2
\end{matrix}\right.$$

Change of triples arise from birational gauge transformations of the bundle fixing the zero section 

\begin{equation} \label{gaugetransform}
\frac{1}{z}=a\frac{1}{z'}+b
\end{equation} 
where $a,b$ are rational functions on $X$, $a\not\equiv0$:
$$\left\{\begin{matrix}
\omega_0'=a\omega_0\hfill\\
\omega_1'=\omega_1-\frac{da}{a}+2b\omega_0\hfill\\
\omega_2'=\frac{1}{a}\left(\omega_2+b\omega_1+b^2\omega_0-db\right)
\end{matrix}\right.$$
A projective structure is thus the data of a triple $(\omega_0,\omega_1,\omega_2)$ up to above equivalence.

\begin{remark}\label{twofirstforms}
Once $\omega_0$ and $\omega_1$ are fixed, $\omega_2$ is completely determined.
\end{remark}
\begin{definition}
The transversely projective  foliation $\F$ has regular singularities if the corresponding connection has at worst regular singularities in the sense of \cite{Deligne}.
\end{definition}
Let us make this more explicit in the case we are concerned with.. 

\noindent Pick a point $x\in X$ and assume first that $x\notin \mbox{Supp}\ N$. Over a small neighborhood of $x$, consider the Ricatti equation $dz+df$ where $f$ is a distinguished first integral of $\mathcal F$. 

Consider now the hypersurface $H=\mbox{Supp}\ N$ along which $\F$ is defined by the logaritmic closed form $\eta_{H}$ . Let us define the Ricatti equation over a neighbohood of $\mathcal H$ as 

\begin{equation}\label{regsing}
\frac{dz}{z+1}+\eta_{\mathcal C}. 
\end{equation}



Note that one recover the original foliation on the section  $z=0$



Let $\{F=0\}$ the local equation at $p$ of the polar locus of $\eta_{\mathcal C}$ corresponding to residues equal to one. Permorming the birationnal tranformation of ${\mathbb P}^1$ bundle $z\rightarrow Fz$, one obtains the new Ricatti  equation $dz+F\eta_{\mathcal C}+z (\eta_{\mathcal C}-\frac{dF}{F})$ without pole on $\{F=0\}$. This latter property enables us to glue together these local models with gluing local bundle automorphisms of the previous form (\ref{gaugetransform}). 
 Note that these transformations preserve the section $\{z=0\}$.  
In this way, we have obtained a ${\mathbb P}^1$ bundle equipped with a Riccati foliation ( in this setting, this means a foliation transverse to the general fiber)
wich induced the sought foliation $\mathcal F$ on a meromorphic (actually holomorphic) section. By \cite{gro}, this ${\mathbb P}^1$ is actually the projectivization of a rank $2$ vector bundle. 

Without additional assumptions on the complex manifold $X$, nothing guarantees that this Riccati foliation  ${\mathbb P}^1$ bundle is the projectivization of a rank $2$ meromorphic flat connection.


However, when $X$ is projective, 
one can assume 
that (birationnaly speaking) this ${\mathbb P}^1$ bundle is trivial and that the resulting Riccati foliation is then defined by a global equation

\begin{equation}\label{riccatiequation}
\Omega=dz+\omega_0+z\omega_1+z^2\omega_2
 \end{equation}

hence defined as the projectivization of the rank two meromorphic flat connection.

$$\nabla\ :\ Z=\begin{pmatrix}z_1\\ z_2\end{pmatrix}\ \mapsto\ dZ+\begin{pmatrix}\alpha&\beta\\ \gamma&-\alpha\end{pmatrix}Z.$$

 It is worth noticing that $\Delta$ has regular singular poles, thanks to the local expression given by  (\ref{regsing}).
 
 Let ${(\nabla)}_\infty$ be this polar locus. From flatness, we inherit a representation of monodromy
 
 $$r_\nabla:\pi_1(X\setminus {(\nabla)}_\infty)\rightarrow SL(2,\C)$$
 whose projectivization gives the  original representation 
 
$$r:\pi_1(X\setminus \mbox{Supp}\ N)\rightarrow PSL(2,\C)$$
 
  \noindent attached to the transversely hyperbolic foliation $\mathcal F$. In particular, one can assume that the image of $r_\nabla$ {\bf lies in $SL(2,\R)$ } (if one takes the Poincar\'e upper half plane model $\mathbb H$ instead of the disk).
 
 \noindent As we are dealing with regular singularities,  note  also than one can recover $\nabla$ from $r_\nabla$ (Riemann-Hilbert correspondance).
 
 \begin{remark}
 As we work up to birationnal equivalence of connection, ${(\nabla)}_\infty$ does not necessarily coincide with $\mbox{Supp}\ N$. We may have added some "fakes" poles around which the local monodromy (associated to $\nabla$) is $\pm Id$
  \end{remark}

Let $U$ a dense Zariski open subset of $X$ wich does not not intersect ${(\nabla)}_\infty$ and denote by $\rho_U: \pi_1(U)\rightarrow SL(2,\mathbb R)$ the  monodromy representation of the flat connection $\nabla$ restricted to $U$.   We want to show that the representation $\rho_\nabla$ does not come from a curve in the following sense: 
\begin{thm}\label{nocurvefactorization}
 Let $\mathcal C$ a quasiprojective curve and let

$$\rho_{\mathcal C}:\pi_1(\mathcal C)\rightarrow PSL(2,\mathbb C)$$
a representation.  Then there is no morphism

 $$\varphi:U\rightarrow \mathcal C$$
  such that $\rho_U$ factors projectively to $\rho_{\mathcal C}$ through $\varphi$.

\end{thm}

\begin{proof}

Up removing some additionnal points in $\mathcal C$ one can realize  $\rho_{\mathcal C}$ as the monodromy  projectivization   of a regular singular rank $2$ connection $\nabla_{\overline{\mathcal C}}$  defined on  the projective closure $\overline{\mathcal C}$. 
Let ${\mathcal R}_X$ the Riccati foliation defined by (\ref{riccatiequation}) and ${\mathcal R}_{\overline{\mathcal C}}$ the one defined by ${\mathbb P}(\nabla_{\overline{\mathcal C}})$. They respectively lie on ${\mathbb P}^1$ bundles $E_X$, $E_{\overline{\mathcal C}}$ over $X$ and $\overline{\mathcal C}$.

We proceed by contradiction. The existence of the morphism $\phi$ can be translated into the existence of a rational map of ${\mathbb P}^1$ bundles $\Psi:E_X\dashrightarrow E_{\overline{\mathcal C}}$ such that

    $${\mathcal R}_X=\Psi^*{\mathcal R}_{\overline{\mathcal C}}.$$
    

In particular, we obtain that ${\Psi_X}^* {\mathcal R}_{\overline{\mathcal C}}=\mathcal F$, where $\Psi_X$ denotes the restriction of $\Psi$ to $X$ (identified with the section $\{z=0\}$. The foliation $\mathcal F$ being a quasi-minimal foliation, this implies that $\Psi_X$ is dominant ( otherwise, $\F$ would admit a rational first integral). Moreover, as the image of $r$ lies in $PSL(2,\R)$ , one can assume that the same holds for $\rho_{\overline{\mathcal C}}$ (this is true up to finite index). In particular none of the leaves  ${\mathcal R}_{\overline{\mathcal C}}$ is dense  (the action of $PSL(2,\mathbb R)$ on ${\mathbb P}^1$ fix a disk). This contradicts the quasi-minimality of $\mathcal F$.


\end{proof}

We are now ready to give the proof of the  main theorem \ref{thprincipal} of the introduction.
\vskip 5 pt
{\it Proof of theorem \ref{thprincipal}}.
By \cite{to}, iv) of Proposition 2.14, the defining form $\eta_H$ has {\it rational} residues. As a by-product,  the local monodromy of the rank two local system defined by $\nabla=0$ around ${(\nabla)}_\infty$ is finite (and more precisely takes values in a group of roots of unity). Keeping in mind theorem \ref{nocurvefactorization}, one can now apply the main theorem of \cite{corsim} (rigid case in the alternative) which asserts that there exists an algebraic morphism $\Phi$ between $X\setminus {(\nabla)}\infty$  and the orbifold space $\frak H$ such that $r=\Phi^*\rho_i$, $1\leq i\leq p=\mbox{Dim}\ \frak h$  where $\rho_i$ is the $i^{\mbox th}$ tautological representation $\Gamma=\pi_1^{orb} (X)\rightarrow \Gamma_i$ defined by the projection onto the $i^{\mbox{th}}$ factor. Actually these representations are related to one another by field automorphisms of $\mathbb C$.

Note that there may be exists other ``tautological'' representations of $\Gamma$ of unitary type (hidden behind the arithmetic nature of $\Gamma$) but they are not relevant in our case, as $r_\nabla$ takes values in $SL(2,\R)$ (see \cite{corsim} for a thorough discussion).



 Actually, the theorem of Corlette and Simpson only claims  a priori the existence of a morphism 
which factorize representations and where foliations are not involved, hence we have to provide an additional argument wich allows passing from representation to foliations, as stated by theorem \ref{thprincipal} . 

To achieve this purpose, consider  the $n$ codimension $1$ foliations ${\F}_j$, $j=1,...,n$ on $\frak H$. These are obtained from the codimension $1$ foliations $dz_j=0$ on the polydisk ${\mathbb D}^p$ after passing to the quotient (this makes sense as $\Gamma$ acts diagonally on ${\mathbb D}^p$). These $p$ foliations give rise by pull-back $p$ transversely hyperbolic foliations $\Phi^* {\F}_j$ on $X\setminus {(\nabla)}_\infty$. Note that $\Phi^*{\F}_i$ has the same monodromy representation than $\F$ but might, a priori differ from this latter.

 Remark that one can recover $\Phi$ from the datum of ${\F}_j$ by choosing in a point $x\in X$ $n$ germs of distinguished first integrals  $Z({\F}_j)$ $,j=1,...p$ attached to ${\F}_j$ and taking their analytic continuation along paths  in $X\setminus {(\nabla)}_\infty$ starting at $x$.. More precisely, the analytic continuation of the $p-uple$ $(Z{({\F}_1)},...,Z({\F}_p)$  defines a map $X\setminus {(\nabla)}_\infty$ to $\frak H$ which is nothing but $\Phi$.

Actually, $\Phi$ extends across ${(\nabla_\infty)}\setminus\mbox{Supp}\ N$ (as an orbifold morphism). This is just a consequence of the Riemann extension theorem, keeping in mind that distinguished first integrals take values in $\mathbb D$ and that the local monodromy around extra poles of ${(\nabla)}_\infty$ is projectively trivial. By the same kind of argument, one can extend $\Phi$ through $\mbox{Supp}\ N$ using that the finiteness of the local monodromy around this latter. The only difference with the previous case lies in the fact that $\Phi$ extends as an analytic map between $X$ and the underlying analytic space of the orbifold $\frak H$, but not as a map in the orbifold setting (this phenomenon occurs when this local monodromy is not projectively trivial).

To get the conclusion, we are going to modify $\Phi$ into an algebraic map
$$\Psi: X\rightarrow \frak H$$
 in order to have $\F=\Psi^*{\F}_i$.

For convenience, $\frak H$ will be regarded as a quotient of the product ${\mathbb H}^p$ of $p$ copies of the upper Poincar\'e's upper halph plane. 

Let $r_i:\pi_1(X\setminus \mbox{supp}\ N)\rightarrow PSL(2,\mathbb R)$ the monodromy representation associated to $\Phi^*{\F}_i$.  One knows that $r$ and $r_i$ coincide up to conjugation in $PSL(2,\C)$. In other words, there exists $\alpha\in PSL(2,\mathbb C)$  such that for every loop $\gamma$ of $\pi_1 (X\setminus \mbox{supp}\ N),\  r_i(\gamma)=\alpha r(\gamma)\alpha^{-1}$.

As the image of $r$ is dense in $PSL(2,\mathbb R)$,  the conjugating transformation $\alpha$ lies in the normalizer  of $PSL(2,\R$ in $PSL(2,\C)$ which is the group generated by $PSL(2,\R)$ and the inversion $\tau (z)=\frac{1}{z}$. We then have two cases to examine, depending whether one can choose $\alpha\in PSL(2,\C)$ or not.

\vskip 5 pt

-  If $\alpha\in PSL(2,\R)$, take $p$ germs of distinguished first integrals associated to $\Phi^*{\F}_j\ j\not=i$ and $\mathcal F$ in $x\in X\setminus {(\nabla)}_\infty$ and consider their analytic continuation along paths.
This induces an analytic orbifold map $X\setminus {(\nabla)}_\infty$ to $\frak H$ which analytically extends to the whole manifold $X$ using the same lines of argumentation  as previously This defines the sought morphism $\Psi$.

\vskip 5 pt

- Otherwise, one can assume that $\alpha=\tau$. We do the same work as before except that we replace the germ of  distinguished first integral $f_x$ of $\mathcal F$ by $\overline {\tau\circ f_x}$. We thus obtain a well defined map 

    $$X\setminus {(\nabla)}_\infty\rightarrow \frak H$$
    
    which is no longer holomorphic (it's antiholomorphic on the "$i^{th}$" component).
    Let $\varphi :{\mathbb H}^p\rightarrow {\mathbb H}^p$ defined by $\varphi (z_1,...,z_p)= (z_1,...,z_i, \overline{\tau}(z_i),z_{i+1},...,z_p)$ and ${\Gamma}'=\varphi\Gamma{\varphi}^{-1}$ . As $\varphi$ is an isometry, ${\Gamma}'$ is still a lattice.  We then obtain, passing to the quotient, a map
    
    $$\Psi_2:\frak H\rightarrow {\frak H}'={\mathbb H}^p/{\Gamma}'$$
    
    One can now easily check that $\Psi=\Psi_2\circ\Psi_1$ (more exactly its extension to the whole $X$) has the required property by changing the complex structure on  $\frak H$ as described above.

\qed



\section{The logarithmic setting}

We begin this section by collecting some classical results about logarithmic $1$ form for further use.
\subsection {Logarithmic one forms}
We follow closely the exposition of \cite{brme}.
Let $X$ a complex manifold and $H$ a {\it simple normal crossing hypersurface }wich decomposes into irreducible components as $H=H_1\cup...H_p$.  A {\it logarithmic 1 form} on $X$ with poles along $H$ is a meromorphic $1$ form $\omega$ with polar set ${(\omega)}_\infty$ such that $\omega$ and $d\omega$ have at most simple poles along $H$. In other words, if $f=0$ is a local reduced equation of $H$, then $f\omega$ and $fd\omega$ are holomorphic. This latter properties are equivalent to say that $f\omega$ and $df\wedge d\omega$ are holomorphic. By localization on arbitrarily small open subset, one obtain a locally free sheaf on $X$ denoted by $\Omega_X^1(\log H)$.

If $(z_1,...,z_n)$ is a local coordinate system around $x\in X$ such that $H$ is locally expressed as $z_1...z_k=0$, then every section of $\Omega_X^1(\log H)$ around $x$ can be written as 

\begin{equation}\label{loglocal}
 \omega=\omega_0+\sum_{i=1}^k g_i\frac{dz_i}{z_i}
\end{equation}

 where $\omega_0$ is a holomorphic $1$ form and each $g_i$ is a holomorphic function.

Equivalently, one can locally write $\omega$ as 

$$\omega=\omega_i +g_i\frac{dz_i}{z_i}$$
where $g_i$ is holomorphic and $\omega_i$ is a local section of $\Omega_X^1(\log H)$ whose polar set does not contain $H_i=\{z_i=0\}$.

Despite that this decomposition fails to be unique, we have a well defined map (the so called {\it Residue map}):

$$\mbox{Res}:\Omega_X^1(\log H)\rightarrow \bigoplus_{i=1}^p{\mathcal O}_{H_i}$$ 

summing on $i$ the maps

$${\mbox{Res}}_{H_i}:\omega\rightarrow {g_i}_{|H_i}.$$

 We will also denote by $\mbox{Res}$ the induced morphism 

$$H^0(\Omega_X^1(\log H)\otimes L)\rightarrow \bigoplus_{i=1}^pH^0(H_i,{\mathcal O}_{L_i})$$ for global logarithmic form twisted by a line bundle $L$ (setting ${\mathcal O}(L_i):={\mathcal O}(L_{|{H_i}})$).

\subsection{A class of twisted logarithmic form}

From now on, we will suppose that there exists on {\bf $X$ K\"ahler} compact, with fixed K\"ahler form $\theta$,   a globally defined twisted logarithmic form, assuming moreover that the dual bundle $L^*$ is {\bf pseudoeffective}. 

Let $h$ be a (maybe singular) hermitian metric on $L^*$ with positive curvature current $\Theta_h\geq 0$. In a local trivialization $L_{|U}\simeq U\times \C$, we thus have

$$ h(x,v)={|v|}^2 e^{-2\varphi (x)}.$$

Hence, $$T=\frac{i}{\pi}\partial\overline{\partial}\varphi$$ is a positive current representing $c_1(L^*)=-c_1(L)$.

and let 
$$\{T\}=\{N\}+Z$$ be  the divisorial Zariski decomposition of this latter.
\vskip 5 pt
In the sequel we will assume that $\mbox{Res}_{H_i} (\omega)\not=0$ for every component $H_i$ of $H$.

\begin{remark}\label{restrictionistrivial}
 
 As $L^*$ is pseudoeffective, one can infer that either $L_i$ is the trivial line bundle, either ${L_i}^*$ is not pseudoeffective and $H_i\subset \mbox{Supp}\ N$.

\end{remark}


As a by-product, one obtains the 

\begin{lemma} \label{alternativeforpoles}

 Let $\mathcal C$ a connected component of $H$, then the following alternative holds:

\begin{enumerate}
 \item[a)] $\mathcal C$ is contained in $\mbox{Supp}\ N$.\label{includeinnegative}
\item[b)] For any irreducible component $H_{\mathcal C}$ of $\mathcal C$, $H_{\mathcal C}\cap \mbox{Supp}\ N=\emptyset$ and $Z.\{H_{\mathcal C}\}=0$. Such a $\mathcal C$ is then called a {\bf non exceptional component}.
\end{enumerate}

\end{lemma}

From now on, we fix a current $T$ representing $c_1(L^*)$.

 By definition of the negative part, $T$ has non vanishing Lelong's numbers along $\mbox{Supp}\ N$. 
 
 In particular, the local potentials $\varphi$ of $T$ are necessarily equal to $-\infty$ on connected component of $H$ satisfying the item a) in lemma \ref{alternativeforpoles}.
 
 Assume now that $H_\mathcal C\subset \mathcal C$ fulfills the properties of point $(2)$. In this case, $L$ is trivial restricted to $H_\mathcal C$ (remark \ref{restrictionistrivial}). Then,  
 
 \begin{enumerate}
\item either the restrictions of $\varphi$ on $H_\mathcal C$ are local {\it pluriharmonic functions},
\item either these restrictions are identically $-\infty$.
\end{enumerate}

Following the previous numerotation, $H_{\mathcal C}$ is said to be of {\bf type $(1)$ or $(2)$}. Note that, by connectedness, every  
irreducible component of $\mathcal C$ has the same type. This allows us to define {\bf connected components of $H$ of type $(1)$ and $(2)$} with respect to $T$, noticing that any connected component belongs to one (and only one) of these types.

\begin{definition}
Let $T$ a closed positive current such that $\{T\}=c_1(N_\F^*)$.

 The boundary divisor $H$ is said to be of {\it type $(1)$} (resp. {\it type $(2)$}) with respect to $T$ if each of its connected components is of type $(1)$ (resp. of type (2)) and when both situations occur, $H$ is said to be of {\it mixed type} (with respect to $T$).
\end{definition}

\begin{lemma}\label{casetype1}
 Let $\mathcal C\subset H$ a connected component of type $(1)$. Assume that $Z\not=0$.
Then there exists a non negative real number $\lambda$  and an integral effective divisor $D$, $\mbox{Supp D}=\mathcal C$ such $\{D\}=\lambda Z$. 
\end{lemma}

\begin{proof}
By pluriharmononicity of the local  $\varphi_{|H}$, $h$ restricts on $H$ to a unitary flat metric $g$. We denote by $\nabla_{g^*}$ the $(1,0)$ part of the Chern connection associated to the dual metric $g^*$. This expresses locally as

  $$ \nabla_{g^*}=\partial\varphi_{|H}+d$$
and $\partial\varphi_{|H}$ is an {\it holomorphic} $1$ form.

Fix $j_0\in\{1,...,p\}$
 The norm $a_{j_0 j_0}={\{\mbox{Res}_{H_{j_0}}(\omega),\{\mbox{Res}_{H_{j_0}}(\omega)\}}_{g^*}$ defined on $H_{j_0}$ is locally given by 

$$ e^{2\varphi_{|H}}{|g_{j_0}|}^2$$

which is obviously {\it plurisubharmonic}. This implies that $a_{j_0 j_0}$  is indeed {\it constant}. Similarly, one can deduce that $a_{j j_0}={\{\mbox{Res}_{H_{j}}(\omega),\mbox{Res}_{H_{j_0}}(\omega)\}}_{g^*}$  is constant along every connected component of $H_{j_0}\cap H_j$.

Locally, one can write $$\omega=\omega=\omega_0+\sum_{j=1}^p g_j\frac{df_j}{f_j}
$$

where $f_{j_0}$ is the local expression of a global section of ${\mathcal O}(H_{j_0})$ vanishing on $H_{j_0}$. Consider a smooth metric on $\mathcal (H_{j_0})$ with local weight $\phi$. One can then easily check that the local forms $$\omega_\phi=\omega_0+\sum_{j\not=j_0} g_i\frac{dz_j}{z_j} +2g_{j_0}\partial \phi$$ glue together along $H_{j_0}$ in a section $\omega_{j_0}$  of $\Omega_{H_{j_0}}^1(\log D_{j_0})\otimes {\mathcal E }^\infty\otimes L_{j_0}$ where $D_{j_0}=\bigcup_{j\not={j_0}}H_{j_0}\cap H_j$  is a simple normal crossing  hypersurface on $H_{j_0}$ and ${\mathcal E}^\infty$ is the sheaf of germs of smooth functions in $H_{j_0}$.

The hermitian product 

$${\{\omega_i, \mbox{Res}_{H_{j_0}}(\omega)\}}_{g^*}= e^{2\varphi_{|H}} \overline{g_{j_0}}\wedge\omega_\phi$$ is well defined as a current on $H_{j_0}$. One may compute its differential with respect to the $\overline{\partial}$ operator. Using that 

$$\frac{1}{2i\pi}\overline{\partial}(\frac{d {z_i}}{ {z_j}})=[z_j=0]\ \mbox{( the integration current along}\ z_j=0)$$ and that the $a_{j j_0}$'s are constant, one easily gets that 

$$\frac{1}{2i\pi}\overline{\partial}({\{\omega_i, \mbox{Res}_{H_{j_0}(\omega)}\}_{g^*})}$$
represents in cohomology the intersection class

$$(\sum_j a_{j j_0}\{H_j\})\{H_{j_0}\}.$$

In particular Stokes theorem yields that $(\sum_j a_{j j_0}\{H_j\}\{H_{j_0}\}){\{\theta\})}^{n-2}=0$. Summing up on all $j_0$, we obtain the following vanishing formula

$$(\sum_{i,j} a_{i j}\{H_i\}\{{H_{j}}\}){\{\theta\}}^{n-2}=0.$$

Thanks to Cauchy-Schwartz's inequality and the fact that $\{H_i\}\{H_j\}{\{\theta\}}^{n-2}\geq 0$ if $i\not=j$, one can promptly deduce that 





$$ \langle I\Lambda,\Lambda\rangle\geq 0$$ where $I$ is the intersection matrix $\{H_i\}\{H_j\}{\{\theta\}}^{n-2}$,  $\langle\rangle$ is the standard Hermitian product on ${\C}^p$ and $\Lambda$ denotes the vector $\left(\begin{array}{c}
 \sqrt {a_{11}}\\
\vdots\\
\sqrt{a_{pp}}
\end{array}
  \right)$.

In particular, $I$ is non-negative and from elementary bilinear algebra , one can produce an effective integral divisor $D$ whose support is $\mathcal C$ and such that 

            $${\{D\}}^2{\{\theta\}}^{n-2}=0.$$
            
 Keeping in mind that we have also $Z^2{\{\theta\}}^{n-2}\geq0$ and $Z\{D\}=0$, one concludes by Hodge's signature theorem that $\{D\}$ is a multiple of $Z$.


\end{proof}

Let $\{\omega, \omega\}_{h^*}$ the norm of $\omega$ with respect to the dual metric $h^*$. This is a well defined positive (a priori non closed) positive current on $X\setminus H$ wich locally expressed as 

$$\eta=\frac{i}{\pi}e^{2\varphi}\omega\wedge\overline{\omega}$$ 

(by a slight abuse of language, we denote also by $\omega$ a local holomorphic $1$ form- the restriction of the global twisted $1$ form $\omega$ to a trivialyzing open set).

Following [De], we inherit from the dual metric $h^*$ a Chern connection  whose $(1,0)$ part $\partial_{h^*}$ acts on $\omega$ as

$$\partial_{h^*} \omega=\partial\omega  + 2\partial\varphi\wedge \omega$$

\noindent on a trivializing chart.

Note that the resulting $(2,0)$ form (well defined outside $H$) has $L_{\mbox{loc}}^1$ coefficients.

\begin{THM}\label{loginvariant}
 Let $X$ be a compact Ka\"hler manifold and $\omega$ a non trivial section of  $\Omega_X^1(\log H)\otimes L$ with $L^*$ pseudoeffectif. 
 Assume that either $Z=0$ ($T$ is then unique and equal to $\left[N\right]$), either $Z>0$ and $H$ is of type (2),
 then, outside the polar locus $H$, the following identity holds (in the sense of current):

\begin{equation}\label{vanishingchernconnection}
\partial_{h^*}\omega=0.
\end{equation}

In particular, $\omega$ is Frobenius integrable ($\omega\wedge d\omega=0$); moreover $ \eta$ and $T=\frac{i}{\pi}\partial\overline{\partial}\varphi$ are closed positive $\mathcal F$-invariant currents respectively on $X\setminus H$ and $X$, where $\F$ is the codimension $1$ foliation defined by $\omega$.
\end{THM}
\begin{remark}

\begin{itemize}

\noindent \item When $H=\emptyset$, this theorem has been established in \cite{de} and this statement is just a generalization to the logarithmic setting of the invariance properties of currents recalled in section \ref{rappelsinvariance}.
\item By invariance of $H$, it is enough to prove the invariance of $T$ on $X\setminus H$ to get the invariance on the whole $X$ (see \cite{to}, Remarque 2.2). 

\end{itemize}

\end{remark}

\vskip 5 pt

The proof of theorem \ref{loginvariant} is divided into two parts, each one corresponding to a particular ``numerical property'' of the positive part $Z$ and the type of the boundary hypersurface $H$. 

\subsection{The case $Z=0$}

\vskip 5 pt
 
  In this situation  $N$ is a $\mathbb Q$ effective divisor. 
  
  Assume firstly that $N$ is an integral effective divisor.


We have   

 $$\omega\in H^0(X,\Omega_X^1 (\log H)(-N)\otimes E)$$ with $E$ a numerically trivial line bundle.

 Consider the  canonical injection 

$$\iota:\Omega_X^1 (\log H)(-N)\otimes E\hookrightarrow \Omega_X^1 (\log H)\otimes E$$ where the first term is regarded as the sheaf of logarithmic $1$ forms with respect to $H$ whose zeroes divisor is contained in $N$.  
 
  Let $\nabla_E$  the unique unitary connection on $E$  and $\nabla_{E_i}$ its restriction on $H_i$. By flatness of  $E_i$, one can infer that 
  
  $$\nabla_{E_i}\mbox{Res}_{H_i}(\iota(\omega))=0$$
 
 As a by-product, and using that $\overline{\partial}$ commutes with $\nabla_E$, one obtains that 
 
       $$\xi=\nabla_E \omega\in H^0(X,\Omega_X^2\otimes E)$$


In particular, one obtains that the differential $d$ of the  well defined $(2,1)$ current  $(\nabla_E\omega)\wedge\overline{\Omega}\wedge{\theta}^{n-2}$ 
is a {\it positive} $(2,2)$ form, namely $-\xi\wedge\overline{\xi}\wedge{\theta}^{n-2}$ . By Stokes theorem, this latter is necessarily identically zero. In other words,

 \begin{equation}\label{vanishingconnection}
 \nabla_E\ \iota (\omega)=0.
 \end{equation}
 
 The proof above is the same of that presented in \cite{brbir}, p.81 (see also \cite{brme} and references therein). We have have just replaced $L$  {\it trivial} by $L$ {\it numerically trivial}.
 
For some suitable open covering $\mathcal U$, $\omega$ is defined by the data of local logarithmic one forms $\omega_U\in H^0(U, \Omega_X^1 \otimes\log H)$, $U\in \mathcal U$ with the glueing condition $f_U\omega_U=g_{UV}f_V\omega_V$, $g_{UV}\in{\mathcal O}^*(U\cap V)$ such that $|g_{UV}|=1$ and $f_U=0$ is a suitably chosen local equation  of $N$.

The vanishing property (\ref{vanishingconnection}) can be rephrased as  
 
  $$\partial {(f_U{\omega_U})}=0\ \mbox{for every}\ U.$$

which finally gives (\ref{vanishingchernconnection}), using that $T=[N]$, the integration current on $N$.

Note that $\partial_{h^*}$ expresses locally as $\partial +\frac{df_U}{f_U}$ is indeed a {\it meromorphic logarithmic flat} connection on $L$.
\vskip 5 pt






Consider now the general case where $N$ is only effective over $\mathbb Q$. By ramified covering trick and after suitable blowng-up, we obtain a k\"ahler manifold $\hat{X}$,  $\mbox{dim}\ \hat{X}=\mbox{dim}\ X$ equipped with a generically finite morphism $\psi:\hat{X}\rightarrow X$ such that $\psi^*(N)$ is an {\it integral} effective divisor and $\hat H =\psi^{-1} (H)$  a normal crossing hypersurface.  One can then assert that $\pi^*\omega\in H^0(\hat X, \Omega_{\hat X}^1 (\log\ \hat{H})\otimes\hat L)$. where $\hat L \in \mbox{Pic}\ X$ is numerically equivalent to $\pi^* (N)$.  Let $\nabla_{\hat h}$ be the logarithmic connection on $\hat L$ defined by a metric $\hat h$ whose curvature current is $-[\psi^* N]$. In other words,  $\nabla_{\hat h}$ is the pull-back by $\psi$ of the logaritmic connection $\nabla_{h^*}$. By the computation performed above, we have $\nabla_{\hat h}\psi^* \omega=0$. For obvious functoriality reasons, this can be restated as $\psi^* (\nabla_{h^*} \omega)=0$, wich eventually leads to $\nabla_{h^*} \omega=0$.\qed 

\subsection{The case $H$ of type (2)}


For $j=1,...p$,  let us fix a smooth metric $g_j$ on $\mathcal O (H_i)$ with local weight $\phi_j$. For $\varepsilon >0$, we thus obtain a smooth metric on $\mathcal O (H_j)$ with local weight ${\phi_j}_{\varepsilon} =\Log{({|f|}^2+\varepsilon e^{2\phi_j})}$.

The $(1,1)$ positive form $\eta:=\frac{i}{\pi}e^{{2\varphi}}\omega\wedge\overline{\omega}$ is  well defined as a current on the whole $X\setminus \mbox{Supp}\ H$. but failed to be integrable along the boundary divisor $H$. To cure this, take a smooth section $\Omega$ of $\Omega_X^{1,0}\otimes L$ (a sheaf in the ${\mathcal C}^\infty$ category), and consider 

   $$\Delta=\frac{i}{\pi}e^{{2\varphi}}\Omega\wedge\overline{\Omega}$$
   
\noindent which is  well defined as a positive current on the whole $X$ (making again the confusion between $\Omega$ and its writing on a trivializing chart)..
   
Let us compute $\overline{\partial}\Delta$ as a current. This latter decomposes as:

$$\pi\overline{\partial}\Delta=A + B$$

where $$A=ie^{2\varphi}\dbar\Omega\wedge\overline{\Omega}$$

and $$B=i((\overline{\partial}e^{2\varphi})\Omega\wedge\overline{\Omega}-e^{2\varphi}\Omega\wedge\dbar\overline{\Omega})$$

\noindent each of the summand being a well defined current on $X$.

\vskip 5 pt




An easy calculation yields

$$i\pi\ddbar{(\Delta)}=i\partial A-2e^{2\varphi}\Omega\wedge\overline{\Omega}\wedge \ddbar\varphi -e^{2\varphi}(2\partial\varphi\wedge \Omega +d\Omega)\wedge\overline{(2\partial\varphi\wedge \Omega+d\Omega)}$$ $$+\dbar(e^{2\varphi}\Omega\wedge\partial\overline\Omega+e^{2\varphi}{\partial\overline\Omega}\wedge\overline{\partial\overline\Omega})$$






Hence, by Stokes theorem,
\begin{equation}\label{decomposeddbar}
\langle i\pi\partial\overline{\partial}\Delta\wedge{\theta}^{n-2}, 1\rangle=0=\langle T_1\wedge{\theta}^{n-2},1\rangle +\langle T_2\wedge{\theta}^{n-2},1 \rangle +\langle\gamma\wedge{\theta}^{n-2},1\rangle
\end{equation}
where $$T_1=-2e^{2\varphi}\Omega\wedge\overline{\Omega}\wedge \ddbar\varphi$$ 

and

$$T_2=-e^{2\varphi}(2\partial\varphi\wedge \Omega+d\Omega)\wedge\overline{(2\partial\varphi\wedge \Omega+d\Omega)}$$

are both positive $(2,2)$ currents (which makes sense, according to \cite{de}) and 

$$\gamma=e^{2\varphi}{\partial\overline\Omega}\wedge\overline{\partial\overline\Omega}.$$ Actually, one performs the same computation as in \cite{de}.except that the final expression contains the residual term $\gamma$ arising from the holomorphicity defect of $\Omega$. 

\vskip 10 pt



Let  $\mathcal U$ an open cover of $X$ such that each $U\in\mathcal U$ is equipped with local coordinates 
$(z_1,..,z_n)$ in such a way that $\omega_{|U}$ can be written as 

\begin{equation}
 \omega_{|U}=\omega_0+\sum_{j=1}^p g_j\frac{df_j}{f_j}
\end{equation}

where $\{f_j=0\}$ is a local equation of $H_j$ and $f_j$ coincide with one of the coordinates $z_k$ whenever $H_j\cap U\not=\emptyset$.

For $\varepsilon>0$, set $$\omega_\varepsilon^U= \omega_0+\sum_{j=1}^p   g_j\varphi_{j,\varepsilon}\frac{df_j}{f_j}$$

where $\varphi_{j,\varepsilon}=\frac{{|f_j|}^2}{{|f_j|}^2+\varepsilon e^{2\phi_j}}$. One can remark that for each $j$, the $\varphi_{j,\varepsilon}$'s glue together on overlapping charts an thus define a global smooth function of $X$. One can see $\omega_\varepsilon^U$ as a smooth section of $\Omega_X^1\otimes L$ which approximates $\omega$ (on $U$) in the following sense

\begin{enumerate}
\item $\lim\limits_{\varepsilon\to 0}\omega_\varepsilon^U=\omega$ weakly as a current.

\item For every $K\Subset U\setminus H$, $\omega_\varepsilon^U$ converges uniformly to $\omega$ on $K$ when $\varepsilon$ goes to $0$.
\end{enumerate}

Before mentioning other properties,
we introduce the following notation:
 
 Let $T$ a $(p,q)$ current  defined on $U$ with $L_{\mbox{loc}}^1$ coefficient written as 

   $$T=\sum_{I,J}F_{I,J} dz_{I}\wedge d\overline{z_{J}}$$

\noindent in standard multiindices notation with respect to the coordinate system $(z_1,..z_n)$.

 We define  

$${\Vert T\Vert}_K=\mbox{sup}_{I,J}{\Vert F_{I,J}\Vert}_{L^1 (K)}$$  the $L^1$ norm being evaluated with respect to the Lebesgue measure $d\mu=|dz_1\wedge d\overline{z_1}....dz_n\wedge d\overline{z_n}|$. 
Now, the key point is to notice that 

\begin{equation}\label{unifbounded}
\sup_{\varepsilon >0}{\Vert\partial({\overline{\omega_\varepsilon^U}})\wedge\overline{\partial({\overline{\omega_\varepsilon^U}}})\Vert}_K<+\infty
 \end{equation}

for every $K\Subset U$.

Indeed, the choice of $\varphi_{i,\varepsilon}$ allows us to "get rid" of the non-integrable terms $\frac{dz_i}{z_i}\wedge\frac{d\overline{z_i}}{\overline{z_i}}$ with could appear after passing to the limit.

Moreover,$\partial{\overline{\omega_\varepsilon^U}}\wedge
\overline{\partial{\overline{\omega_\varepsilon}}}$ 
converges uniformly to $0$ on every $K\Subset U\setminus H$.

Let $(\psi_U)$ be a partition of unity subordinate to the open cover $\mathcal U$. Set 

$$\Omega_\varepsilon=\sum_U \psi_U \omega_\varepsilon^U.$$

Note that this defines a smooth section of $\Omega_X^1 \otimes {L}$ whose restriction to $U$ share the same properties than ${\omega_\varepsilon}^U$. 


In particular, we get
\begin{equation}
\lim\limits_{\varepsilon\to 0}e^{2\varphi}\partial{\overline{\Omega_\varepsilon}}\wedge\overline{\partial{\overline{\Omega_\varepsilon}}}=0\ \mbox{in the weak sense}
\end{equation}

Actually this is obvious on $K\Subset X$ (where the convergence is even uniform) and this extends to the whole $X$, thanks to  (\ref{unifbounded}) and the fact that $\varphi=-\infty$ on $U\cap H$ and is upper-semicontinuous (as a plurisubharmonic function).

In order to conclude, replace $\Omega$ by $\Omega_\varepsilon$ in (\ref{decomposeddbar} ) and $\gamma,\  T_1,\  T_2$ by the corresponding $\gamma_\varepsilon,\  {T_1}_\varepsilon,\  {T_2}_\varepsilon$, keeping in mind that $\lim\limits_{\varepsilon\to 0}\gamma_\varepsilon=0$, and that $T_{1\varepsilon}$ and $T_{2\varepsilon}$ are positive current wich converges on $X\setminus H$ respectively to 

$$S_1= -2e^{2\varphi}\omega\wedge\overline{\omega}\wedge \ddbar\varphi$$

and

$$S_2=-e^{2\varphi}(2\partial\varphi\wedge \omega+d\omega)\wedge\overline{(2\partial\varphi\wedge \omega+d\omega)}$$

Passing to the limit when $\varepsilon$ goes to $0$, one can infer that $S_1=S_2=0$, whence the result.\qed

\begin{cor}\label{integrabilityforlogarithmic}
  Let $X$ a compact K\"ahler manifold and $\omega\in H^0(\Omega_X^1(\log H)\otimes L)$ a non trivial twisted logaritmic $1$ form with poles on a simple normal crossing hypersurface $H$.  Assume that $L^*$ is pseudo-effective. Then $\omega$ is integrable and there exists an $\F$ invariant positive current representing $L^*$, where $\F$ is the foliation defined by $\omega$. In particular, the integration current $[N]$ is $\F$ invariant, where $N$ is the negative part in the Zariski's divisorial decomposition of $c_1(L^*)$.
  \end{cor}
  
  \begin{proof}One can assume that $\omega$ has non trivial residues along the irreducible components of $H$. The only remaining case to deal with is $Z\not=0$ and among the connected components of $\mathcal C$, there exists at least one of type (1) with respect to a given closed positive current $T$ representing $c_1(N_\F^*)$. By making use of lemme \ref{casetype1}, one can change $T$ into another representing positive current $T'$ such that $H$ becomes type (2) with respect to $T'$.
\end{proof}

\section{Pseudoeffective logarithmic conormal bundle and invariant metrics}

Let $(X, \mathcal D)$ be a pair consisting of a compact k\"ahler manifold $X$ equipped with a codimension $1$ holomorphic distribution $\mathcal D$. We assume that there exists a $\mathcal D$-invariant hypersurface $H$ (i.e: $H$ is tangent to $\mathcal D$)  such that 

$$L=N_{\mathcal D}^*\otimes \mathcal O (H)$$ is {\it psef}. 

From the normal crossing and invariance conditions, $\mathcal D$ is defined by a twisted logarithmic one form $\omega\in H^0(X, \Omega_X^1(\log{H})\otimes L) $ with non vanishing residues along $H$ (see for instance Lemma 3.1 in \cite{brme}).

As a consequence of corollary \ref{integrabilityforlogarithmic}, $\mathcal D$ is {\bf integrable} and we will call $\F$ the corresponding foliation.

 The following definitions are borrowed from the terminology used in the study of $\log$ singular varieties, even if the setting is somewhat different.

\begin{definition}
Let $\F$ a foliation like above. Let $H=H_1\cup ....\cup H_p$ the decomposition of $H$ into irreducible components. 

$\F$  is said to be of 
\begin{enumerate}
\item {\bf KLT} type if there exists $p$ rational numbers $0 \leq {a_i} < 1$ such that $$c_1(N_{\F}^*)+\sum_i a_i\{H_i\}$$ is a {\it pseudoeffective} class.
\item {\bf strict Log-canonical} type if the the condition above is not satisfied. 
\end{enumerate}

\end{definition}

\subsection{Existence of special invariant transverse metrics for KLT foliations}

 Here, we assume that the foliation $\F$  is of KLT type. This means (see definition above) that $c_1(N_{\mathcal F}^*)+\sum_i\lambda_i \{H_i\}$ is {\it pseudoeffective} where the $\lambda_i$'s are real numbers such that $0\leq\lambda_i<1$.
 
\begin{definition}
Set $\Lambda=(\lambda_1,...,\lambda_p)$ and $H_\Lambda= \sum\lambda_i H_i$ .

$\Lambda$ will be called a {\bf KLT datum}.

\end{definition}










By definition of pseudoeffectiveness, there exists on $N_{\F}^ *$ a singular metric $h$ wich locally can be expressed as 

    $$h(x,v)={|v|}^2 e^{-2\varphi (x)+\sum 2\lambda_i\log{|fi|}}$$

where $\varphi$ is a locally defined {\it psh} function such that $$T_\Lambda=\frac{i}{\pi}\ddbar \varphi$$ is a globally defined positive current representing $c_1(N_{\mathcal F}^*)+\{D\}$, the $\{f_i=0\}$'s are local reduced equation of $H_i$ and $\omega$ is a local defining one form of $\F$. On $N_\mathcal D$, the dual metric $h^*$ can be represented by the positive $(1,1)$ form $$\eta_{T_\Lambda}=\frac{i}{\pi}e^{2\varphi-2\sum_i\lambda_i\log{|f_i|}}\omega\wedge\overline{\omega}$$ 

\noindent with locally integrable coefficients, which is unique up to multiplication by a positive constant once $T_{\Lambda}$ is has been fixed (cf \cite{to} Remarque 1.4, Remarque 1.5).

Let ${\mathcal C}_\F$ the cone of closed positive current $T_\Lambda$ such that $\{T_\Lambda\}=c_1(N_{\mathcal F}^*)+\{H_\Lambda\}$.

 \noindent We are going to establish a similar statement to that of Theorem \ref{transmetric}

\begin{thm}\label{transmetricklt}
Let $\{N_\Lambda\}+Z_\Lambda$ be the Zariski decomposition of the pseudoeffective class  $c_1(\con)+\{H_\Lambda\}$ and $T_\Lambda\in{\mathcal C}_\F$.

Then, both $T_\Lambda$ and $\eta_{T_\Lambda}$ are {\bf closed $\F$ invariant positive currents}. Moreover, $T_\Lambda$ can be chosen in such a way that  
\begin{enumerate}

\item $T_\Lambda=[N_\Lambda]$ if $Z_\Lambda=0$ ({\bf euclidean type}).
\item $T_{\Lambda}=[N]+{\eta_{T_\Lambda}}$ if $Z_\Lambda\not=0$ ({\bf hyperbolic type})
($\eta_{T_\Lambda}$ being normalized such that $\{\eta_{T_\Lambda}\}=Z_\Lambda$).
\end{enumerate}

\end{thm} 


\begin{proof}
Let $T_\Lambda$ in ${\mathcal C}_\F$. The closed positive current $S= T_\Lambda+\sum_i (1-\lambda_i )[H_i]$ represents $c_1( \con )+\{H\}$.  As a by-product of theorem \ref{loginvariant}, we obtain that

  $$(2\partial\varphi-2\sum\lambda_i\frac{df_i}{f_i})\wedge\omega=-d\omega=0$$
  
 \noindent which implies the closedness and $\F$ invariance of $T_\Lambda$ and $\eta_{T_\Lambda}$. 
  
  The remainder follows the same line of argumentation as in \cite {to} (which corresponds to $\Lambda=0$). From $\F$-invariance of $T_\Lambda$, one can deduce that $[N_\Lambda]$ is also ${\F}$- invariant and that the same holds for any closed positive current representing $Z_\Lambda$.  Note that the theorem is proved when $Z_\Lambda=0$.

  We investigate now the nature of the singularities of the foliations.
  Outside $H$, $\eta_T$ has a local expression of the form $ie^{\Psi}\omega\wedge\overline{\omega}$ with $\Psi$ {\it psh} and thus admits local {\it elementary first integrals} (Th\'eor\`eme 2 of \cite{to}). Near a point on $H$, the same holds replacing $\eta_T$ by $\pi^*\eta_T$  where $\pi$ is a suitable local branched covering over $H$ ( we use that $\lambda_i<1$) and one can then again conclude that $\F$ admits a local elementary first integral ( this property is clearly invariant under ramified covering).
  
  By proposition \ref{intersectionresidual}, $\{\eta_{T_\Lambda}\}\{W\}=0$ for any closed $\F$ invariant positive current $W$. In particular, this holds whenever $W=\eta_{T_\lambda}$ or $\{W\}=Z_\Lambda$. By Hodge's signature theorem, one can then infer that $Z_\Lambda$ is a multiple of $\{\eta_{T_\Lambda}\}$ and after normalization, that $\{\eta_{T_\Lambda}\}=Z_{\Lambda}$ whenever $Z_\Lambda \not=0$. In this case, we have 

       $$\{T_\Lambda\}=\{N_\Lambda\}+\{\eta_{T_{\Lambda}}\}\label{cohomologicallevel}.$$

From now on, we assume that $Z_{\Lambda}\not=0$. By the normalization process above, $\eta_T$ is then uniquely determined by the datum of $T_{\Lambda}\in{\mathcal C}_\F$.

Let ${\mathcal C}_{Z_\Lambda}$ the cone of closed positive currents $T$ such that $\{T\}=Z_{\Lambda}$.

By Banach-Alaoglu's theorem, ${\mathcal C}_{Z_{\Lambda}}$ is compact with respect to the weak topology on current.

 From \ref{cohomologicallevel}, one inherits a map

   $$\beta:{\mathcal C}_{Z_{\Lambda}}\rightarrow{\mathcal C}_{{Z_{\Lambda}}}$$

\noindent defined by $\beta (T)=\eta_{T+[N_\Lambda]}$.
Remark that the point $(2)$ ot the theorem is equivalent to the existence of a fix point for $\beta$. This fix point will be produced by the Leray-Schauder-Tychonoff's theorem once we have proved that $\beta$ is {\it continuous}. This can be done following the same approach than \cite{to}(D\'emonstration du lemme 3.1, p.371). We briefly recall the idea (see {\it  loc.cit} for details): Let $(T_n)\in  
{{\mathcal C}_{Z_\Lambda}}^{\mathbb N}$ converging to $T$ such that the sequence $\beta(T_n)$ is convergent. On has to check that 
$$\lim\limits_{n\to +\infty} \beta(T_n)=\beta(T).$$
 Locally, one can write $T_n=\frac{i}{\pi}\ddbar\varphi_n$ and $\beta(T_n)=\frac{i}{\pi}e^{2\varphi_n}\frac{\omega\wedge\overline{\omega}}{\prod_i{|f_i|}^{2\lambda_i}}$. By plurisubharmonicity, one can suppose, up extracting a subsequence, that $(\varphi_n)$ is uniformly bounded from above on compact sets and converges in $L_{loc}^1$ to a {\it psh} function $\varphi$ wich necessarily satisfies

$$\frac{i}{\pi}\ddbar\varphi=T.$$

By Lebesgue's dominated convergence, one obtains that $e^{2\varphi_n}\frac{\omega\wedge\overline{\omega}}{\prod_i{|f_i|}^{2\lambda_i}}$ converges in $L_{ loc}^1$ to $e^{2\varphi}\frac{\omega\wedge\overline{\omega}}{\prod_i{|f_i|}^{2\lambda_i}}$. This obviously implies that

$$\lim\limits_{n\to +\infty}\beta(T_n)=\beta(T)$$.
\end{proof}

The main informations provided by the proof above are:

\begin{itemize}

\item The $\F$-invariant positive current $\eta_{T_\Lambda}$ represents the positive part $Z_\Lambda$ when this latter is non $0$.

\item For every $\F$-invariant positive current $S$, one has $\{\eta_{T_\Lambda}\}\{S\}=0$, in particular ${\{\eta_{T_\Lambda}\}}^2=0.$

\end{itemize}

This is enought to get an analogous statement to that of \cite{to}, Proposition 2.14 and Corollaire 2.15.

 \begin{prop}\label{Z(alpha)nef}Let $\F$ a KLT foliation with KLT datum $\Lambda=(\lambda_1,...,\lambda_p)$. Let $\Xi$ a $(1,1)$ positive current $\mathcal F$-invariant (for instance, $\Xi=T_\Lambda$). Let's consider the Zariski's decomposition.  
 
$$  \alpha =\{N(\alpha)\}+Z(\alpha)$$

with $\alpha=\{\Xi\}$ (e.g: $\alpha=c_1(N_{\mathcal F}^*)+ \{H_{\Lambda}\}).$

Then the following properties hold:
\begin{enumerate}[a)]
\item  the components  $D_i$ of the negative part $N(\alpha)$ are hypersurfaces invariant by the foliation;  in particular,  $Z(\alpha)$ can be represented by an $\F$ invariant closed positive current.

\item  $Z(\alpha)$ is a multiple of $\{\eta_T\}$.

\item  \label{square=0} $Z(\alpha)$ is {\it nef} 
 and ${Z(\alpha)}^2=0$.

\item  The decomposition is orthogonal: $\{N(\alpha)\} Z(\alpha)=0$. More precisely, for every component $D_i$ of $N(\alpha)$, one has $ \{D_i\} Z(\alpha)=0$.
\item   In $H^{1,1}(M,\R)$, the  $\R$ vector space spanned by the components $\{D_i\}$ of $\{N(\alpha)\}$ intersects the real line spanned by $\{\eta_T\}$ only at the origin.
\item  The decomosition is rational if $X$ is projective and $\alpha$ is a rational class (e.g: $\alpha=c_1(N_{\mathcal F}^*)+\{H_\Lambda\} $).
\item  Every  $(1,1)$ closed positive current wich represents $\alpha$ is necessarily $\mathcal F$-invariant.
\item \label{familleexcep} let $A$ be a hypersurface invariant by the foliation and $A_1,...,A_r$ its irreducible components; the family $\{A_1,...,A_r\}$ is exceptional if and only if the matrix $(m_{ij})=(\{A_i\}\{A_j\}{\{\theta\}}^{n-2})$ is negative, where $\theta$ is a K\" ahler form on $X$.
 \end{enumerate}

\end{prop}

\begin{proof}
This is essentially a consequence of the Hodge's signature theorem (see,\cite{to} for the details).
\end{proof}

In this setting one can introduce, exactly on the same way as  \cite{to} D\'efinition 5.2 (see also section \ref{psefcon}), the locally constant sheaves ${\mathcal I}^\varepsilon$ and ${\mathcal I }_{d\log}^\varepsilon$ attached to the transverse metric structure given by theorem \ref{transmetricklt}. Strictly speaking, those sheaves are only defined on the complement of $H\cup\mbox{Supp}\ N_\Lambda$.

Let $\mathcal P$ the set of prime divisor of $X$. To each $P\in\mathcal P$, one can associate non negative numbers $\lambda_D$ , $\nu_D$ defined by the decompositions

      $$N_\Lambda= \sum_{D\in\mathcal P}\nu_D D. $$
      $$H_\Lambda= \sum_{D\in\mathcal P}\lambda_D D. $$

Following again exactly the same lines of argumentation of \cite{to} p.381-384, one can establish the following proposition similar to proposition 5.1 and proposition 5.2 of {\it loc.cit}:

\begin{prop}\label{logformnearK} Let $K_1,...,K_r$ be a finite family of irreducible hypersurfaces invariant by $\F$. Set $K=K_1\cup K_2...\cup K_r$.  Then there exist on a sufficiently small connected neighborhood of $K\cup \mbox{Sing}\  (\mathcal F)$ a uniquely defined section  $\eta_K$ of $\ladm$ such that the polar locus ${(\eta_K)}_\infty$ contains 
$K\cup \mbox{Sing}\  (\mathcal F)$  with the following additional properties:

 \begin{enumerate}
 \item For every $i=1,...,r$, the residue of $\eta_K$ along $K_i$ is equal to $1+\nu_{K_i}-\lambda_{K_i}$ and equal to $1$ on the other components of ${(\eta_K)}_\infty$.  In particular the residues of $\eta_K$ are {\bf non-negative} real numbers.
 
 \item If $K$ is connected, $\{K_1,...., K_r\}$ is an exceptional family if and only if the germ of  ${(\eta_K)}_\infty$ along $K$ does not coincide with $K$. 
 \end{enumerate}
\end{prop}

\begin{remark}
As the foliation only admits singularities of elementary type,  the set of separatrices at any point $x\in K\cup \mbox{Sing}\ \F$ coincide locally with the poles of $\eta_K$ 
  (see \cite{to}, Remarque 4.1).
\end{remark}
\subsection{Existence of special invariant transverse metrics for strict Log-Canonical foliations}
As usual,  $\{N\}+Z$ will denote the Zariski's divisorial decomposition of $c_1(N_{\F}^*)+\{H\}$.

\subsubsection{The case Z=0}\label{caseZ=0logcan}
In this situation,   $c_1(N_{\F}^*)\otimes \mathcal{O} (H)$ is represented by a unique positive current, namely $T=[N]$.
 
  Let $H_b$ the union of non exceptional components (see lemma \ref{alternativeforpoles}). $H_b$ will be called the {\bf boundary part} of $H$. Note that $H_b$ is non empty, otherwise the foliation would be a KLT one.  On $X\setminus H_b$,  $\F$ admits an invariant transverse metric $\eta $.   The semi-positive $(1,1)$ form $\eta$ involved in the statement of theorem \ref{loginvariant} is well defined as a current on $X\setminus H_b$. Indeed, it can be expressed locally as 
  
  $$\eta=\frac{i}{\pi}e^{2\psi}\frac{\tilde{\omega}\wedge\overline{\tilde{\omega}}}{{|F|}^2\prod_i {|f_i|}^{2\mu_i}}$$
  
  where the $f_i=0$'s are  defining equations of  the components $H_i$ of $H$   contained in $\mbox{Supp}\ N$ , with $0\leq\mu_i<1$ , $F=0$ is a  defining equation of $H_b$, $\psi$ is a  plurisubharmonic function such that $\frac{i}{\pi}\ddbar \psi=[N]-\sum_i\mu_i[H_i]$ 
  and $\tilde{\omega}$ is a local generator of $N_{\F}^*$ (i.e: a local defining holomorphic $1$ form of $\F$ without  zeroes divisor). On a neighborhood $V_x$ of a point $x\in H_b$ where $z_1,..z_{n_x}=0$ is a local equation of $H_b$, the closedness property (\ref{vanishingchernconnection}) shows that the foliation is defined by a closed logarithmic form of the form 

$$\omega=\sum_{i=1}^ {n_x}\alpha_i \frac{dz_i}{z_i}$$  
and that 
\begin{equation}
\eta=\frac{i}{\pi}\omega\wedge\overline{\omega}\label{euclideanatinfinity}.
\end{equation}

(for suitable choice of $z_i$'s) where the $\alpha_i$'s are non zero complex numbers.

 Because of this local expression, one can observe that ${\mbox{Sing}\ \F}_b:=\mbox{Sing}\ \F\setminus H_b$ is a {\it compact} analytic subset of $X\setminus H_b$. 
 Contrarily to the KLT case, this current $\eta$ does not  extend to $X$ for the simple reason that it can not be extended through $H_b$. 
 
 The archetypal example is provided by logarithmic foliations on the projective space ${\mathbb P}^n$, that is foliations defined by a logarithmic form $\omega$. Such a form is automatically closed whenever the polar locus ${(\omega)}_\infty$ is normal crossing and thus induces the euclidean transverse metric $\eta= \frac{i}{\pi}\omega\wedge\overline{\omega}$, only defined outside the polar locus. In this example, it is also worh noticing that $H_b$ is ample, a phenomenon wich does not occur  for KLT foliations: in this situation, the class $\{D\}$  of any divisor whose support is $\F$-invariant has numerical dimenson at most $1$, thanks to  proposition \ref{Z(alpha)nef}. Another important difference lies in the fact that the number of $\F$-invariant hypersurfaces is arbitrarily hight (compare with proposition \ref{boundedfamilyKLT} ).
 
 Coming back to our setting, $\eta$ defined an euclidean transverse structure to wich one can associate as before a locally constant sheaf of distinguished first integrals  $\Ieucl$ and the corresponding sheaf of logarithmic differentials $\leucl$. These latters are defined a priori on the complement of $\mbox{Supp}\ N\cup H$. Near a point $x\in H_b$, the foliation admits as first integral the multivaluate section $\sum_i\alpha_i\log {z_i}$ of $\Ieucl$ whereas these multivaluate sections take the form of elementary first integral $\prod f_i^{\lambda_i}$, $\lambda_i >0$ at a neighborhood of $x\in\mbox{Supp}\ N$.

 Thanks to this euclidean structure, one can easily  exhibit a section $\eta_K$ of  $\Ieucl$ on a neighborhood of $K\cup{\mbox{Sing}\ \F}_b$ where $K=\mbox{Supp}\ N$ whose polar locus ${(\eta_K)}_\infty$  contains $\mbox{Supp}\ N$. 
 
 \begin{lemma}\label{etaklogZ=0}
 ${(\eta_K)}_\infty$ does not coincide with $K$ along $K$. 
 \end{lemma}
 \begin{proof}
 Suppose that equality holds. Then 
 
 $${(\sum_i \lambda_i \{N_i\})}^2=0$$
 
 \noindent where the $N_i$'s stand for the irreducible components of $\mbox{Supp}\ N$ and $\lambda_i>0$ is the residue of $\eta_N$ along $N_i$.
  Moreover, $F=|e^{\int \eta_N}|$  is a function defined near $\mbox{Supp}\ N$, constant on the leaves and such that the family $\{F<\varepsilon\}$ forms a basis of neighborhoods of $\mbox{Supp}\ N$.  
 
 Let $\varepsilon>0$  and and $\psi:\R\rightarrow [0,\infty)$ continuous non identically zero such that $\psi$ vanishes outside $(0,\varepsilon)$. Set $g=\psi\circ F$ (this makes sense is $\varepsilon$ is chosen small enought). This guarantees that $g\eta$ is an $\F$ invariant  closed positive current  $S$ such that $\{S\}\{N_i\}=\{S\}^2=0$, thanks to remark \ref{elementaryinsupport}. By Hodge's signature theorem, $\sum_i\lambda_i\{N_i\}$ is proportional to $\{S\}$. This latter being a {\it nef} class ($S$ has no positive Lelong numbers), we get the sought contradiction.
 \end{proof}
\begin{definition}
 Let $V$ an open neighborhood of $H_b$; an hermitian metric defined on $V\setminus H_b$ is said to be {\it complete at infinity} if there exists an open neighborhood $U$ of $H_b$, $U\Subset V$ such that for every maximal geodesic arc $\gamma :[0,a) \rightarrow V$, 

\begin{itemize}
 \item either $a=+\infty$,
\item either there exists $0<t<a$ with $\gamma(t)\in V\setminus U$.
\end{itemize}

\end{definition}

\begin{lemma}\label{bundlelikemetricatinfinity}
 There exists a bundle like metric on $X\setminus  (H_b\cup\mbox{Supp}\ N\cup \mbox{Sing}\ \F)$ inducing the invariant transverse metric $\eta$ wich is complete at infinity.
\end{lemma}

\begin{proof}
Keeping the notation above, for $x\in H_b$, one put on $V_x$ the metric with poles $g_x=i\sum_{l=1}^{n_x}\frac{dz_l}{z_l}\wedge\frac{d\overline{z_l}}{\overline{z_l}} +h_x$ where $h_x$ is a smooth hermitian metric on $V_x$.  Let $h$ a smooth hermitian metric on $X\setminus H_b$ which is the restriction of an hermitian metric on $X$.. By partition of unity subordinate to the cover of $X$ determined by $X\setminus H_b$ and the collection of $V_x$, these models glue together in a complete hermitian metric defined on $X\setminus H_b$. Thanks to the expression of $\eta$ given in (\ref{euclideanatinfinity}), this latter is conformally equivalent on $X\setminus  (H_b\cup\mbox{Supp}\ N\cup \mbox{Sing}\ \F)$ to a metric $g$ fulfilling the conclusion of the lemma.

\end{proof}
\subsubsection{The case $Z\not=0$}.\label{strictZnot0}
As we shall see later, this situation is described by the theorem \ref{thprincipallog} of the introduction.
 We first collect some observations useful in the sequel. 
 
\begin{lemma}\label{negativeintersectionform} Let $\mathcal C$ a connected component of $H$  which is not contained in $\mbox{Supp}\ N.$ and let  ${\mathcal C}_1,...,{\mathcal C}_r$ be its irreducible components. 

Then, the intersection form $\{{\mathcal C}_i\} \{{\mathcal C}_j\} \{\theta\}^{n-2}$ is negative definite and in particular $\{{\mathcal C}_1,....,{\mathcal C}_r\}$  forms an exceptional family.
 \end{lemma}
 \begin{proof} Suppose by contradiction that the intersection matrix is not negative.
 Then, thanks to remark \ref{alternativeforpoles} and Hodge's signature theorem, one can exhibit $r$ non-negative real numbers $\mu_1,...,\mu_r$ with at least one positive, such that $\sum_{i=1}^r\mu_i \{{\mathcal C}_i\}$ is colinear to $Z$.
 
 
 In particular, one can represents $Z$ by the integration current $\lambda [D]$, where $\lambda>0$ and $D=\sum_{i=1}^r\mu_i {\mathcal C}_i$. With the help of lemma \ref{casetype1}, this implies that for any other connected component ${\mathcal C}^{'}$ of $H$ not supported on $\mbox{Supp}\ N$, one can extract an effective divisor $D'$ supported exactly on ${\mathcal C}^{'}$ such that $\{D^{'}\}$ and $Z$ are colinear. Then $\F$ turns out to be a KLT foliation: a contradiction.
 \end{proof}
 
 Like the previous situation, we will denote by $H_b$ the union of these connected components and will call it the {\it boundary component}.

 \begin{lemma}
 Let $T$ a closed  positive current representing $Z$. Then $T$ has no Lelong numbers along $H_b$ and the local psh potentials of $T$ are necessarily equal to $-\infty$ on $H$.
 
 \end{lemma} 
 
 \begin{proof}
 
 The fact that the local potentials are $-\infty$ on $H_b$ is a consequence of lemma \ref{casetype1} together with lemma \ref{negativeintersectionform}. In particular, $T$ is necessarily invariant according to theorem \ref{loginvariant}. Let $T$ be a positive closed current, $\{T\}=Z$.  Assume by contradiction that $T$ has some non vanishing belong numbers on a connected component $\mathcal C$ of $H_b$. Let   ${\mathcal C}_1,...,{\mathcal C}_r$ the irreducible components of $\mathcal C$, and $\lambda_i$ the Lelong number of $T$ along ${\mathcal C}_i$. We first claim that there doesn't exist  two indices $i\not=j$ such that $\lambda_i=0$, $\lambda_j\not=0$ and ${\mathcal C}_i\cap {\mathcal C}_j\not=0$. Otherwise this would imply that $Z\{{\mathcal C}_i \}\{{\theta}^{n-2}\}>0$, a contradiction with lemma \ref{alternativeforpoles}. One can then easily deduce that $\lambda_i\not=0$ for each $i$.  
Let 

           $$T=\sum_{i=1}^{+\infty} {\mu}_i [K_i] +R$$
           
           \noindent the Siu's decomposition of $T$. 
           
           Remark, that from the $\F$- invariance of $T$, one can infer that each hypersurface $K_i$ is also $\F$ invariant whenever $\mu_i\not=0$ and that the same holds for the residual part $R$.

By virtue of theorem \ref{loginvariant}, $\F$ admits in a neighborhood of $\mathcal C$  an invariant current $\eta=\eta_T$ wich can be locally written

$$\eta=\frac{i}{\pi}e^{2\psi}\frac{{\omega}\wedge\overline{{\omega}}}{\prod_j {|f_j|}^{2\alpha_j}}$$

with $f_j=0$ a reduced  equation of ${\mathcal C}_j$, $\omega$ a  defining form of $\F$,  $0\leq \alpha_j<1$ and $\psi$ a {\it psh} function.

This implies (see the proof of theorem \ref{transmetricklt}) that the singularties of the foliation near $\mathcal C$ are of elementary type, more precisely of the form 

   $F=u{f_1}^{\beta_1}...{f_r}^{\beta_r}$

where $u$ is some unit and the $\beta_i$'s are non negative real numbers. 
Suppose that for some $i$,  $K_i\nsubseteq \mathcal C$. By remark 4.1 of \cite{to}, $K_i\cap \mathcal C=\emptyset$ and in particular $\{K_i\}\{{\mathcal C}_j\}=0$ for every $j$. Moreover, $\{R\}\{{\mathcal C}_j\}=0$ according to remark \ref{elementaryinsupport}. hence, the equality 

    $$Z\{C_j\}=0$$

can be reformulated as

$$(\sum_i\lambda_i\{{\mathcal C}_i\})\{{\mathcal C}_j\}=0$$

wich contradicts lemma \ref{negativeintersectionform}.
 
 \end{proof}
 
By theorem \ref{loginvariant}, $\eta=\eta_T$ is a well defined as a closed positive current on $X\setminus H_b$ but a priori not extend through $H_b$ as a current.
We will actually show that this extension really holds but this requires further analysis and the proof is postponed to the last section. Moreover, we will explain how we can ensure that the invariant transverse metric associated to $\eta$ is hyperbolic, a key property to reach the conclusion of theorem \ref{thprincipallog}.


\section{Settling the case of KLT foliations}

In this section, we aim at discussing abundance properties of the logarithmic  conormal bundle $N_{\F}^*\otimes \mathcal O(H)$ for KLT foliations (defined in the previous chapter). The ambient manifold is again assumed to be {\it K\"ahler compact}. 

Here, $\mbox{Kod}(L)$ stands for the {\it Kodaira dimension} of a line bundle $L$ and $\nu (L)$ for its {\it numerical Kodaira dimension}.

One firstly recall the statement of Bogomolov's theorem ( see \cite{brme} and references therein).

\begin{prop}  
Let $X$ a compact K\"ahler manifold X and $H\subset X$ a normal crossing hypersurface. Assume that for some $L\in \mbox{Pic} (X)$, there exists a non trivial twisted logarithmic form
 $\omega\in H^0(X, \Omega_X^1(\log H)\otimes L)$. Then 
  $\mbox{Kod} (L^*)\leq 1$; moreover, when equality holds, $\omega$ is integrable and the corresponding foliation is a meromorphic fibration.
 \end{prop}
 
 \begin{cor}
 Let $\mathcal F$ a codimension $1$ foliation on a compact K\"ahler manifold $X$. Let $H$ a normal crossing hypersurface invariant by $\F$. Then  $\mbox{Kod} (N_{\F}^ * \otimes \mathcal{O} (H))\leq 1$ and $\F$ is a meromorphic fibration when equality holds.
\end{cor}

By virtue of corollary \ref{integrabilityforlogarithmic} and item \ref{square=0}) of proposition \ref{Z(alpha)nef}, these upper bound remains valid in the numerical setting:
\begin{prop}
 Let $\mathcal F$ a codimension $1$ foliation on a projective manifold $X$. Let $H$ a normal crossing hypersurface invariant by $\F$. Then  $\nu (N_{\F}^ * \otimes \mathcal{O} (H))\leq 1$, where $\nu$ stands for the numerical dimension.\end{prop}

 \begin{definition}
 Let $\Lambda=(\lambda_1,...,\lambda_p)$ a KLT datum, $A\subset X$ a prime divisor and $\mathcal C$  a connected component of $H$.  
 
 \begin{itemize}
 \item We will denote by $m(A)$ the multiplicity of $A$ along $H_\Lambda$: $m(A)=\lambda_i$ if $A=H_i$, $m(A)=0$ otherwise.
 \item $\mathcal C$ is said to be exceptional if its irreducible components form an exceptional family.
 \end{itemize}
 \end{definition}
 
 \begin{lemma}\label{exceptionalsuppN}
  Let $\mathcal C$ a connected component of $H$, then the following properties are equivalent
 
 \begin{itemize}
 \item $\mathcal C$ is exceptional.
 \item $\mathcal C$ is contained in $\mbox{Supp}\ N$.
 \end{itemize}
 
 \end{lemma}
 
 \begin{proof}
 Assume that $\mathcal C$ is exceptional and suppose by contradiction that $\mathcal C$ is not contained in $\mbox{Supp}\ N$. This implies, using remark \ref{restrictionistrivial} , that $\mathcal C\cap H$ is empty and $L=N_{\F}^*\otimes \mathcal O (H)$ is trivial in restriction to each  irreducible component of $\mathcal C$. 
 The foliation $\F$ is defined by a logarithmic form $\omega\in H^0(X, \Omega(\log H )\otimes L)$ whose residues along the irreducible components of $\mathcal C$ are  
  nowhere vanishing. Pick a point $x\in \mathcal C$ with meets $n_p$ branches of $\mathcal C$ and such that  ${\eta_{\mathcal C}}_\infty$ has extra components (not supported in $\mathcal C$) at $p$ (this is made possible by proposition \ref{logformnearK})  and considered a local holomorphic coordinate patch $z=(z_1,...,z_n),\ z(x)=0$ such the equation of $\mathcal C$ is given by $z_1...z_{n_x}=0$ and such that the  logarithmic  form $\eta_{\mathcal C}$ is expressed near $x$ as 
  
  $$\eta_{\mathcal C}=\sum_{i=1}^{n_x}\alpha_i\frac{dz_i}{z_i} +\frac{df}{f}$$
  where $f=0$ is a reduced equation of the additional poles and the $\alpha_i$'s are {\it positive} real numbers.  On the other hand, 

$$\omega= \omega_0+\sum_{i=1}^{n_x}g_i\frac{dz_i}{z_i}$$ in a trivializing neighborhood of $p$ where $\omega_0$ is an holomorphic $1$ form and the $g_i$'s are units. As $\omega$ and $\eta_{\mathcal C}$ define the same foliation, there exists a meromorphic function $U$ such that $U\eta_{\mathcal C}=\omega$. One can then easily check that $U=fV$ with $V$ holomorphic and this forces the residue $g_i$ to vanish on $\{z_i=0\}\cap \{f=0\}$: a contradiction. The converse implication is obvious.
 \end{proof}
 
 \subsection{The case $Z\not=0$}.

  Let $\{N\}+Z$ be  the Zariski decomposition of $c_1(N_{\F}^* ) + \{H \}$ and assume that $Z\not=0$.

Fix an arbitrarily small $\varepsilon>0$ and
consider an arbitrary non exceptional component $\mathcal C$ of $H$.  Up renumerotation of indices, one can write $\mathcal C=H_1\cup...\cup H_l,\ l\leq p$.  Keeping trak of proposition \ref{Z(alpha)nef}, one can extract from $\mathcal C$, using Hodge's signature theorem, a $\mathbb Q$ effective divisor  $D=\sum_{i=1}^l\lambda_i H_i$ with the following properties
 \begin{enumerate}
 \item For every $1\leq i\leq l$, one has $0<\lambda_i<1$.
 \item there exists 
a positive integer $p$ such that $\lambda_{1}=\frac{1}{p}$.
 \item there exists a real number $\lambda$, $0<\lambda<\varepsilon$ 
such that $\{D\}=\lambda Z$.
 \end{enumerate}
 Combining these observations with lemma \ref{exceptionalsuppN}, one obtains the following statement:
 
\begin{prop}Let $\F$ be a KLT foliation and assume that  $Z\not=0$. 

Then one can find a KLT datum $\Lambda$  with enjoys the following properties:

\begin{enumerate}

\item  $\mbox{Supp}(N_\Lambda)=\mbox{Supp}(N)$.
\item  $Z_{\Lambda}$ is non trivial multiple of $Z$.
\item For any non exceptional connected component $\mathcal C\subset H$, there exists a prime divisor $H_\mathcal C \subset \mathcal C$ and a positive integer $p$ tel que $m(H_\mathcal C)=1-\frac{1}{p}$.
\end{enumerate}

\end{prop}

We are now ready to prove the main theorem of this section (wich corresponds to theorem \ref{thprincipallog} of the introduction).

\begin{THM}
 Let $\F$ be a KLT foliation on $X$ projective. Assume moreover that $Z\not=0$. Then,
\begin{enumerate}
\item Either $\F$ is tangent to the fibers of an holomorphic map $f:X\rightarrow S$ onto a Riemann surface and the abundance principle holds: $$\mbox{Kod}(N_{\mathcal F}^*\otimes\mathcal O (H))=\nu (N_{\mathcal F}^*\otimes\mathcal O (H))=1.$$ This case occurs in particular whenever there exists at least one non exceptional component $\mathcal C\in H$,

\item
either $\mathcal F={\Psi}^*\mathcal G$ where $\Psi$ is a morphism of analytic varieties between $X$ and the quotient $\frak{H}={{\D}^n}/{\Gamma}$ of a polydisk, ($n\geq 2)$ by an irreducible lattice $\Gamma\subset {(\mbox{Aut}\ \D)}^n$ and $\mathcal G$ is one of the $n$ tautological foliations on $\frak{H}$.

\end{enumerate}

\end{THM}


%

\begin{proof}
 Let $\Lambda$ be a KLT datum satisfying the properties given by the previous proposition and consider the transversely hyperbolic structure attached to this latter.
 The strategy remains essentially the same as in the proof of theorem \ref{thprincipal}. In particular, the first step consists in relating the monodromy of the foliation inherited from the hyperbolic  transverse structure  with its dynamical properties. Assume firstly that each component of $H$ is exceptional and therefore it is contained in $\mbox{Supp N}$. In this case, we do exactly the same  and prove the same as the situation "$N_{\F}^*$ pseudoeffective", starting from the fibers connectedness of the developing map studied in paragraph \ref{fibers connectedness}, then establishing the same results proved in sections \ref{kodaira}  and \ref{nonabundant} wich finally leads to the same conclusion.
 
 One needs to modify slightly the previous construction when there exists non exceptional components. Indeed, on such a component $\mathcal C$, the hyperbolic transverse metric is likely to degenerate; moreover $\mathcal C$ is not connected by any exterior leaf in the sense that every separatrix of $\F$ passing through $x\in\mathcal C$ is a local branch of $\mathcal C$ ( consequence of proposition \ref{logformnearK}). In section \ref{psefcon} , one had to remove $\mbox{supp}\ (N)$ (where the transverse metric also degenerates) to define properly the developing map $\rho$. However, if one looks thoroughly at the arguments developped in the same section, the fact that $\rho$ is both complete and has  "connected fibers" (in the sense of theorem \ref{connect}) heavily relies on the existence of connecting exterior leaf  in $\mbox{supp}\ N$.  If one comes back to our current setting, the developing map is now defined on some covering of $X\setminus (\mbox{supp} N\cup H)$.  The point is that one can extend the basis of this covering and in that way get rid of the inexistence of connecting leaves. This can be done as follows: Let $\{{\mathcal C}_1,...,{\mathcal C}_r\}$ the set of non exceptional components. For each ${\mathcal C}_i$, select  a prime divisor $H_{{\mathcal C}_i} \subset \mathcal C$  such that $m(H_{{\mathcal C}_i })=1-\frac{1}{p_i}$  for some positive integer $p_i$. By proposition \ref{logformnearK}, $\F$ admits locally  around each $H_{{\mathcal C}_i}$ and away from $\mbox{Sing}\ \F$ a (multivaluate) section of $\Ihyp$ of the form ${f_i}^{\frac{1}{p_i}}$ where $f_i=0$ is a suitable reduced equation of $H_{{\mathcal C}_i}$. Set $\tilde{H_i}= H_{{\mathcal C}_i}\setminus\mbox{Sing}\ \F$ and $A=\bigcup \tilde{H_i}$ (note that $\mbox{Sing}\ \F\cap H_{{\mathcal C}_i}$ is precisely the union of intersection loci of $H_{{\mathcal C}_i}$ with the other components of ${\mathcal C}_i$). 

Using this observation, one can construct an infinite Galoisian ramified covering
 
 $$\pi:X_0\rightarrow (X\setminus (H\cup\mbox{Supp}\ N))\cup A. $$

such that $\pi^*\Ihyp$ becomes a constant sheaf and wich ramified exactly over each $H_{{\mathcal C}_i}$ with order $p_i$. In particular the developping map $\rho$ (defined as a section of $\pi^*\Ihyp$) is submersive along $\pi^{-1}(A)$. This allows us to prove that $\rho$ is complete (its image is the whole Poincar\'e disk) using exactly the same arguments as \cite{to}, th\'eor\`eme 3.

Keeping the same notations as subsection \ref{fibers connectedness}, one defines a binary relation $\mathcal R$ on $X_0$ setting $p\mathcal R q$ if
\begin{enumerate}
 \item ${\mathcal C}_p={\mathcal C}_q$

or

\item $\rho (p)=\rho (q)$ and there exists a connected component $K$ of $\mbox{Supp}\ N$ such that ${(\eta_K)}_\infty\cap\pi ({\mathcal C}_p)\not=\emptyset$ and ${(\eta_K)}_\infty\cap \pi({\mathcal C}_q)\not=\emptyset$ near $K$

 or 
\item There exists $i\in\{1,...,r\}$ such $\rho (p)=\rho (q)\in H_{{\mathcal C}_i}$   
\end{enumerate}

 We denote by $\overline{\mathcal R}$ the equivalence relation generated by $\mathcal R$.
 
 \noindent Similarly to proposition \ref{fibers connectedness} , one can prove that the map $\overline{\rho}: X/\overline{\mathcal R }$ onto $\D$ is one to one. The notion of modified leaf remains the same as definition \ref{modifiedleaf} except that we include now as modified leaves each non exceptional component. Taking into account this slight modificaton the statement of theorem \ref{leafdynamic} remains valid. In particular, one can conclude that $\F$ is a fibration provided  there exists at least one non exceptional component. 
 
 Following the same line of argumentaton as in the begining of the proof of theorem \ref{abundancecases}, one can also conclude that $\mbox{Kod}(N_{\mathcal F}^*\otimes\mathcal O (H))=1$.
 
 \end{proof}
\subsection{The case $Z=0$}
Here, we aim at proving the analogue of theorem \ref{abundancecases} item 1), namely

\begin{thm}

 Let $\F$ be a KLT foliation on $X$ projective. Assume moreover that $Z=0$. Then,
$$\mbox{Kod}(N_{\mathcal F}^*\otimes\mathcal O (H))=\nu (N_{\mathcal F}^*\otimes\mathcal O (H))=0.$$ 
\end{thm}

\begin{proof}
In this situation, $N$ is an effectve divisor over $\mathbb Q$. From the fact that the integration current $[N]$ is the only closed positive current wich represents the class $\{N\}$, one easily infers that $H$ is contained in $\mbox{Supp}\ N$. Let $N_i$, $i=1,...,l$ the irreducible components of $\mbox{Supp}\ N$. One can then exhibit a KLT datum $\Lambda=(\lambda_1,...,\lambda_p)\in {\mathbb Q }^p$ such that 
$$c_1(N_{\F}^*)+\{H_\Lambda\}= \{N_\Lambda\}$$

\noindent where $N_\Lambda\leq N$ is $\mathbb Q$ effective. This KLT datum induces a transverse euclidean metric in the sense of theorem \ref{transmetricklt}. Let $G\subset\Im(\C)$ the monodromy group attached to this tranverse structure (recall that is defined as the image of the representation $r:\ \pi_1(X\setminus \mbox{Supp}\ N)\rightarrow \Im (\C))$ associated to the locally constant sheaf $\Ieucl$.
One proceed on the same way that the proof of theorem \ref{abundancecases}  ($\varepsilon=0$): When the generic leaf of $\F$ is compact, the linear part of $G$ is finite and abundance holds for $ N_{\mathcal F}^*\otimes\mathcal O (H)$. Otherwise, by performing a suitable covering ramified over $\mbox{Supp}\ N$, one can suppose, after desingularization, that the foliation is given by an {\it holomorphic} section of $\Omega_X^1\otimes L$  (thanks to the fact that $N_\lambda-H_\lambda=\sum_i \mu_i N_i$ where each $\mu_i$ is a rational $>-1$) where $L$ is a flat line bundle such that $\mbox{Kod}\ (N_{\mathcal F}^*\otimes\mathcal O (H))=0$ whenever $L$ is torsion. This actually holds for exactly the same reasons that {\it loc.cit}.

\end{proof}

Note also that the statement of corollary \ref{boundedfamily} remains valid:

\begin{prop}\label{boundedfamilyKLT}
 Let  $\F$ a KLT foliation on a compact k\" ahler manifold $X$. Assume that $\mathcal F$ is not a fibration, then the family of irreducible hypersurfaces invariant by $\mathcal F$ is exceptional and thus has cardinal bounded from above by the Picard number $\rho (X)$.
\end{prop}
 

 






  
  \section{Dealing with strict Log-Canonical foliations}
  
  \subsection{The case $Z=0$}\label{caseZ=0}.
  
  The next result shows that the abundance principle holds true for this class of foliations.
  
  \begin{THM}
  Let $\F$ a strict log-canonical foliation on a projective manifold $X$. Assume moreover that $Z=0$. Then the Kodaira dimension of $N_{\F}^*\otimes {\mathcal O} (H)$ is equal to zero.
  \end{THM}
  
  \begin{proof}
  We will make use of informations collected in paragraph \ref{caseZ=0logcan}. We will also closely follow the presentation  made in \cite{coupe} as well as some results of {\it loc.cit} (especially paragraph 3.2, section 4 and references therein).
  From the existence of the euclidean transverse structure (namely, the locally constant sheaf $\Ieucl$), one inherits a representation 
  
      $$\varphi:\Gamma\rightarrow \Im{(\C)}$$
      
      setting $\Gamma=\pi_1(X\setminus (\mbox{Supp}\ N\cup H))$.  
      
      If $\gamma$ belongs to $\Gamma$, then we can write $\varphi(\gamma)(z)=\varphi_{L}(\gamma) +\tau(\gamma)$ where $$\varphi_L:\Gamma\rightarrow S^1\subset {\C}^*=GL(1,\C)$$ is a homomorphism and 
      
           $$\tau:\Gamma\rightarrow \C$$
           is a $1$-cocycle  with values in ${\C}_{\varphi_L}$, that is $\C$ with its structure of $\Gamma$ module  induced by the linear representation $\varphi_L$.  
   
   As usual, we will call the image  $G$ of $\varphi$ the monodromy group of the foliation. 
   Clearly, abundance holds whenever the linear part $G_L=\mbox{Im}(\varphi_L)$ of $G$ is finite.

   
As previously, one can defined a developping map $\rho$ attached to $\Ieucl$, that is $\rho$ is a section of $\pi^*\Ieucl$ where $\pi:X_0\rightarrow X\setminus (\mbox{Supp}\ N\cup H))$ is a suitable Galoisian covering.

One can then define the same equivalence relation $\overline{\mathcal R}$ as in section \ref{fibers connectedness} and conclude similarly that the induced map $$\overline{\rho}:X_0/\overline{\mathcal R}\rightarrow \C$$ 

\noindent is one to one and onto. Actually the proof proceeds on the same way, using lemma \ref{etaklogZ=0} and replacing compactness of $X$ by the existence of a complete bundle-like metric, as stated by lemma \ref{bundlelikemetricatinfinity}.

Like in section \ref{sectiondynamics}, this enables us to claim that $G$ faithfully encodes the dynamic of $\F$ and in particular that $G_L$ is finite whenever the topological closure of leaves are hypersurfaces of $X$.

       Henceforth, we proceed by contradiction assuming $|G_L|=+\infty$. 
         One can notice in addition that small loops around the components of $H_b$ give rise to infinite additive monodromy. In other words, the subgroup of translations of $G$ is also infinite. In particular, $G$ has no fixed points. At the cohomological level, this means that
     
     $$H^1(\Gamma,{\C}_{\varphi_L})\not=0.$$
     
     From a result directly borrowed from \cite{coupe} (Theorem 4.1), one can deduce that there exists a morphism $f$ of $X\setminus (\mbox{Supp}\ N\cup H))$ onto an orbifold curve $C$ and a representation ${\varphi}_C$ of the orbifold fundamental group of $C$ in $\Im{(\C)}$  such that the following diagram commutes

$$\xymatrix{ 
    \Gamma \ar[r]^\varphi \ar[d]_{f_*} & \Im{(\C)} \\ \pi_1^{orb} (C) \ar[ru]^{{\varphi}_C}  }$$
     
  Up performing some suitable blowing-up at infinity, one can assume that $f$ extends to an holomorphic fibration $$\tilde{f}:X\rightarrow\overline{C}.$$
Looking at the diagram above, one can observe that the sheaf $\Ieucl$ extends through the neighborhood of a generic fiber $\tilde{F}$ as a constant sheaf. In particular, the foliation admits an {\it holomorphic} first integral on such a neighborhood. This latter is necessarily constant on $\tilde{F}$ by compactness of the fibers. Hence, one obtains that $\F$ is tangent to the fibration, whence the contradiction.
\end{proof}




\vskip 7 pt

\subsection{The case $Z\not=0$}

This last possible situation requires some little bit more involved analysis.
Let $\mathcal C\subset H_b$  a boundary connected component and ${D}_1,...,{D}_l$ its irreducible component admitting repectively the reduced equations $f_1=0,..,f_l=0$.

On a neighborhood of ${D}_k$ the foliation is represented by a twisted one form $\omega^k$ valued in a line bundle trivial on ${D}_k$, wich locally expresses as $$\omega^k={\omega_0} +\sum_{i=1}^l\lambda_i\frac{df_i}{f_i}$$ where $\omega_0$ is holomorphic and  for each $i$, $\lambda_i$ is a non vanishing holomorphic functions on $f_i=0$.  

For notational convenience, we will assume that $D_k\cap D_j$ is connected (maybe empty) for every $k,l$.
Note that the quotient of residues $\lambda_{kj}=\frac{\lambda_k}{\lambda_j}$ on the non empty intersections ${D}_k\cap{D}_j$ is intrinsically defined as a non zero complex number and also that $\lambda_{lk}=\frac{1}{\lambda_{kl}}$ .Setting $\lambda_{kl}=0$ when ${D}_k\cap{D}_l=\emptyset$, this leads to the following intersection property:

    $$\mbox{For every}\ k,\ {\sum_{j=1}^l\lambda_{kj}D_j}.D_k=0.\label{vanishingintersection}$$

\begin{lemma}\label{cyclecomponents}
 Let $k\in\{1,...,l\}$. Then, there exists $p\geq 2$ integers $\{i_1,...,i_p\}\in\{1,...,l\}$ such that 

\begin{enumerate}

\item $k=i_1=i_p$
\item $D_{i_j}\cap D_{i_{j-1}}\not=\emptyset $ for every $2\leq j\leq p$.
\item $\prod_{j=1}^{p-1}\lambda_{{i_j}{i_{j+1}} }\not=1$.\label{productnot=1}
\end{enumerate}
\end{lemma}

\begin{proof}
Assume the contrary. This means that there exists $l$ non zero complex numbers $\lambda_1,  \lambda_l$ such that $\lambda_{ij}=\frac{\lambda_i}{\lambda_j}$ as soon as $D_i\cap D_j\not=\emptyset$. Hence, for every $j$, we obtain that

$$(\sum_{i=1}^l \lambda_i \{D_i\})\{D_j\}=0$$

which contradicts the negativity of the intersection matrix $\{D_i\}\{D_j\}\{\theta^{n-2}\}$.
\end{proof}

\subsubsection{Reduction to dimension 2}

We begin by the following local statement. 
\begin{prop} Let $\F$ a foliation defined on $({\mathbb C}^n,0)$ by some meromorphic form $\omega=\omega_0+\lambda_1\frac{dz_1}{z_1}+\lambda_2\frac{dz_2}{z_2}$, with $\omega_0$ holomorphic, $\lambda_1,\lambda_2\in {\C}^*$. Then there exists a germ of submersive holomorphic map $u=(u_1,u_2)$ of $({\mathbb C}^n,0)$ onto $({\mathbb C}^2,0)$ such that $\F=u^* {\mathcal F}'$ where ${\mathcal F}'$ is a foliation on $({\mathbb C}^2,0)$ defined by ${\omega}'={\omega_0}'+\lambda_1\frac{du_1}{u_1}+\lambda_2\frac{du_2}{u_2}$. Moreover, there exists units $v_i$ such that $u_i=z_iv_i$.

\end{prop}

\begin{proof}
Suppose that $n>2$. Consider the holomorphic $1$ form $\Omega=z_1z_2\omega=z_1z_2\omega_0+\lambda_1z_2dz_1+\lambda_ 2z_1 dz_2$. Note that the vanishing locus $\mbox{Sing}\ (\Omega)$ of $\Omega$ is precisely the codimension $2$ analytic subset $\{z_1=0\}\cap \{z_2=0\}$. One can write $\omega_0=\sum_i a_i dz_i$ where the $a_i$'s are holomorphic functions and remark that 
$Z=\frac{\partial}{\partial z_n}-\frac{z_1 a_n}{\lambda_1}\frac{\partial}{\partial z_1}$ is a non vanishing vector field  belonging to the kernel of $\Omega$ and tangent to $\mbox{Sing}\ (\Omega)$. Consequently, there exists a diffeomorphism $\phi\in\mbox{Diff}({\C}^n,0)$ preserving the axis $\{z_1=0\}$ and $\{z_2=0\}$ such that $\phi_* Z=\frac{\partial}{\partial z_n}$. One can thus assume that $ Z=\frac{\partial}{\partial z_n}$. Keeping in mind that $\Omega$ is Frobenius integrable ($\Omega\wedge d\Omega=0$), this easily implies, up multiplying by a suitable unit, that $\Omega$ does not depend on the $z_n$ variable. We then obtain the result by induction on $n$.
\end{proof}

Pick a point $x\in \mathcal C$ contained in the pure codimension $2$ locus of $\mbox{Sing}\ \F$.
i.e: $x$ belongs to some intersection $D_i\cap D_j$ and $x\notin D_k$ for each $k\notin\{i,j\}$.  

Following the previous proposition, in a suitable local coordinates system $z=(z_1,...,z_n)$, $z(x)=0$, $D_i=\{z_1=0\}$, $D_j=\{z_2=0\}$ the foliation $\F$ is defined by an holomorphic one form only depending of the $(z_1, z_2)$ variable 

$$\Omega=\lambda_1 z_2dz_1+\lambda_2z_1dz_2 + z_2z_1\omega_0$$

with $\lambda_1$, $\lambda_2$ two non zero complex numbers whose quotient $\frac{\lambda_1}{\lambda_2}$ is precisely $\lambda_{ij}$ and $\omega_0$ an holomorphic one form.




Thanks to theorem \ref{loginvariant} and \ref {strictZnot0},   $\F$ admits near $p$ two invariant currents $T$ and $\eta$  (recall that this latter is a priori only defined in the complement of $D_i\cup D_j=\{z_1z_2=0\}$). 

It is worth keeping in mind that the local potentials of $T$ are automatically equal to $-\infty$ on $D_i\cup D_j$ despite the fact that $T$ has no non vanishing Lelong numbers along $D_i\cup D_j$.


We are therefore reduced to study a germ  foliation  in $({\C}^2,0)$ defined by a logarithmic form 
$$\omega=\omega_0+\lambda_1\frac{dz_1}{z_1}+\lambda_2\frac{dz_2}{z_2}$$

with non vanishing residues $\lambda_1$, $\lambda_2$ such that there exists in addition a {\it psh} function $\psi$ satisfying the following properties:

\begin{enumerate}
\item $T=\frac{i}{\pi}\ddbar \psi$ is an $\F$ invariant current. Moreover  $\psi=-\infty$ on the axis $z_1z_2=0$ but $T$ has no atomic part (i.e:Lelong numbers) along $\{z_1z_2=0\}$.\label{-inftyonaxis}
\item $\eta=\frac{i}{\pi}e^{2\psi}\omega\wedge\overline{\omega}$ is an $\F$ invariant current in the complement of $\{z_1z_2=0\}$.

\end{enumerate}
\begin{lemma}
Let $\F$ a germ a foliation in $({\C}^2,0)$ satisfying the above properties. Then $\F$ is linearizable; more precisely, there exists a germ a biholomorphism $\Phi\in\mbox{Diff}({\C}^2,0)$ preserving each axis such that $\Phi^*\F$ is defined by the form $\lambda_1\frac{dz_1}{z_1}+\lambda_2\frac{dz_2}{z_2}$.  Moreover, $\lambda_{ij}=\frac{\lambda_1}{\lambda_2}$ is {\bf positive real number}.
\end{lemma}

\begin{proof}
Putting the two properties mentioned in (\ref{-inftyonaxis})  together, one can infer that $T$ is a positive current  without atomic part on the separatrices $\{z_1z_2=0\}$ which is non zero on every neighborhood of $x$. According to \cite{br}, this implies that $\frac{\lambda_1}{\lambda 2}$ is a real number.
One has to discuss the following cases wich might occur and wich depend on the value of $\lambda=\frac{\lambda_1}{\lambda_2}$.

\begin{itemize}
\item $\lambda<0$. One knows (see \cite{brbir}) that $\F$ is linearizable; indeed, either $-\lambda\notin {\mathbb N}^* \cup{1\over{\mathbb N}^*}$ and then belongs to the Poincar\'e domain, either the foliation  singularity admits a Poincar\'e Dulac's normal form which is linearizable thanks to the existence of two separatrices. In particular, one can assume that $\F$ is defined by $\omega= \lambda_1\frac{dz_1}{z_1}+\lambda_2\frac{dz_2}{z_2}$, a closed logarithmic $1$ form. By assumptions, there exists $\psi$ {\it psh} such that $\eta= \frac{i}{\pi}e^{2\psi}\omega\wedge\overline{\omega}$ is closed as a current. This easily implies that $\psi$ is actually constant on the leaves. 

Note also that $F=z_1{z_2}^\lambda$ is a multivaluate first integral of $\F$ with connected fibers on $U\setminus\{z_2=0\}$ where $U$ is a suitable neighborhood of the origine and that $F(U\setminus\{z_2=0\}=\C$ (thanks to $\lambda<0$). Therefore, there exists a subharmonic function $\tilde{\psi}$ defined on $\C$ such that $\tilde{\psi}\circ F=\psi$.  Keeping in mind that a bounded subharmonic function on $\C$ is actually constant and applying this principle to $e^{\tilde{\psi}}$, one obtains that $\psi$ is constant, wich obviously contadicts property (\ref{-inftyonaxis}).
\item $\lambda>0$. The linearizability of the foliation  is equivalent to that of the germ of holonomy diffeomorphism 

      $$h(z)=e^{2i\pi\lambda}z+\ hot,\ \lambda=\frac{\lambda_1}{\lambda_2}$$ evaluated on a tranversal $\mathcal T=(\C,0)$ of $\{z_1=0\}$ (see \cite{MM}). This one does not hold automatically and is related to diophantine property of $\lambda$ (Brujno's conditions). 
 
 Assume that $h$ is non linearizable and let $\D_r=\{|z|< r\}$, $r>0$ small enought.
 We want to show that this leads to a contradiction with the existence of the invariant current $\eta$.
  Following Perez-Marco (\cite{pe}), there exists a totally invariant compact subset  $K_r$ with empty interior (the so-called hedgehog) 
  which verifies the following properties:  
 \begin{enumerate}
 \item $K_r$ is totally invariant by $h$: $h(K_r)=K_r$ and is maximal with respect to this property.
 \item $K_r\cap\partial{{\D}_r}\not=\emptyset$.
 
 \end{enumerate}
 
 \vskip 5 pt
 Let $0<r<r'$ for which the above properties hold.
 
The following argument is adapted from  (\cite{brsurf}, preuve du lemme 11, p.590) .\\

\noindent 

 \noindent Let $\gamma$  be a rectifiable curve contained in  $\mathcal T\setminus\{0\}$. its length with respect to the metric form $g=e^{\varphi}\frac{|dz|}{{|z|}}$ is given by

$$l_g(\gamma)=\int_\gamma \frac{e^{\varphi (z)}}{{|z|}}d\mathcal H$$

\noindent where $\mathcal H$ denotes the dimension 1 Hausdorff measure.



 Because that $\eta$ is invariant by the foliation, one gets

$$l_g(\gamma)=l_g(h\circ\gamma).$$

Using the existence and properties of $K_r$, $K_{r'}$ stated above, one can easily produce a sequence of rectifiable curves

         $$\gamma_n:[0,1]\rightarrow \{r\leq |z| \leq r'\}$$
  
  
 \noindent with $|\gamma_n(0)|=r, |\gamma_n(1)|=r'$ and such that  $l_g(\gamma_n)$ converges to $0$. Up extracting a subsequence, $\gamma_n$ converge, with respect to the Hausdorff distance, to a compact connected subset $K$ not reduced to a point contained in a polar subset, which provides, like in \cite{brsurf}, {\it loc.cit} the sought contradiction.

         

\end{itemize}

\end{proof}

\begin{remark}
We have shown that every singularity of $\F$ is actually elementary.
\end{remark}


\subsubsection{Dulac's transform and extension of the positive current  $\eta_T$ extends through the boundary components}

Let $\F$ a foliation given on a neighborhood of $0\in {\C}^2$ by the $1$ logarithmic form 

$$\omega=\lambda_1\frac{dz_1}{z_1}+\lambda_2\frac{dz_2}{z_2},\  \lambda_1\lambda_2>0\ (\mbox{Siegel domain})$$

Let $T_1$ be the transversal $\{z_2=1\}$ and $T_2$ the transversal $\{z_1=1\}$. The foliation induces a multivaluate holonomy map, the {\it Dulac's transform}, defined as

$$\begin{array}{cc} d_\lambda:& T_1\rightarrow T_2
\\&z_1 \rightarrow z_2={z_1}^\lambda\end{array}$$

where $\lambda=\frac{\lambda_1}{\lambda_2}$.

One now makes use of property \ref{productnot=1} stated in lemma \ref{cyclecomponents} above.
  Let $D$ a component of $\mathcal C$. Up  renumerotation, one can assume that 

$D=D_1$ and that there exists $D_2, ..., D_p$. such that  $D_{i}\cap D_{i-1}\not=\emptyset $  for every $2\leq i\leq p$, $D_1\cap D_p\not=\emptyset$. $\prod_{j=1}^{p-1}\lambda_{{i_j}{i_{j+1}} }\not=1$.

For every $i\in\{1,...,p\}$, let  $T_i\simeq (\C,0)$ a germ a transversal to $\F$ in a regular point $m_i$ of $D_i$.


Equip $(T_1,m_1)$ with an holomorphic coordinate $z$, $z(m_1)=0$ . Denote by $S_\eta$ the invariant positive current $\eta_T$ restricted to $T_1$. For any Borel subset $B$ of $T_1$, let $\nu_B (S_\eta)$ be the mass of $S_\eta$ on $B$, i.e: $\nu_B (S_\eta)=\int_B\frac {i}{\pi} \frac{e^{2\varphi (z)}dz\wedge d\overline z}{{|z|}^2}$. Note that this mass is finite whenever $B$ is compact subset of $T_1\setminus {\{ m_1\}} $.

When composing Dulac's transforms, one get a multivaluate holonomy tranformation of $T_1$ wich can be written as $h(z)= z^r U(z)$ where $r=\lambda_{12}\lambda_{23}.....\lambda_{p1}$  is positive real number strictly greater than $1$ (up reversing the order of composition) and $U$ is a bounded unit on each angular sector such that $\lim\limits_{z\to 0}U(z)$ exists and equal to $1$ up doing a change of variables $z\rightarrow \lambda z$. Up iterating $h$ and thus replacing $r$ by $r^n$, one can also assume that $r>2$. 

By a slight abuse of notation, we will identify $h$ with one of its determinations on a sector of angle $<2\pi$, that is $S_{\theta_1,\theta_2}= \{z\in T_1|\theta_1<\arg z<\theta_2,\ 0<\theta_2-\theta_1< 2\pi\}$. 

Consider $\rho>0$ small enought and for every integer $n>0$, let $A_n\in T_1$ be the annulus $\{\frac{\rho^{r^{n+1}}}{2}<|z|<2\rho^{r^n}\}$.

For $\theta=\theta_2-\theta_1\in ]0,\pi[$, consider the annular sector $A_{n,\theta_1,\theta_2}= S_{\theta_1,\theta_2}\cap A_n$.

One can easily check that $h(A_{n,\theta_1,\theta_2})$ contains an annular sector of the form $A_{n+1,{\theta_1}',{\theta_2}'}$ with ${\theta_2}'-{\theta_1}'>2({\theta_2}-{\theta_1})$. 

From these observations and the fact that $S_\eta$ is invariant by $h$, one deduces that $\nu_{A_n} (S_\eta)\leq \frac{1}{2^{n-1}}\nu_{A_1} (S_\eta)$. As $\{0<|z|<\rho\}=\bigcup_n A_n$, $S_\eta$ has finite mass on  $T_1\setminus\{0\}$.    

Consequently, the closed positive current  $\eta_T$ well defined only a priori on $X\setminus H_b$ extends trivially on the whole $X$ as a closed positive current. This extension will be also denoted by $\eta_T$.

\subsubsection{Existence of an invariant transverse hyperbolic metric, monodromy behaviour at infinity and consequences on the dynamics}

Every singularity of the foliation has elementary type and then, according to proposition \ref{intersectionresidual} and remark \ref{extensiontocodimension3},

$$\{\eta_T\}\{S\}=0$$

\noindent for every $\F$ invariant closed positive current $S$.

The positive part $Z$ of $c_1(\con \otimes\mathcal O (H))$ is represented by an $\F$ invariant positive current, one can then  make use of Hodge's index theorem to conclude that 

$$\{\eta_T\}=Z$$
up normalization of $\eta_T$ by a multiplicative positive number.

This last property enables us to give an analogous statement to that of proposition \ref{Z(alpha)nef}:

\begin{prop}\label{Z(alpha)neflogcan}Let $\F$ a strict Log-Canonical  foliation on a K\"ahler manifold $X$. Suppose moreover that $Z\not=0$. Let $\Xi$ a $(1,1)$ positive current $\mathcal F$-invariant. 
 
$$  \alpha =\{N(\alpha)\}+Z(\alpha)$$

avec $\alpha=\{\Xi\}$ (for instance, $\alpha=c_1(N_{\mathcal F}^*)+ \{H\}$

Then the following properties hold:
\begin{enumerate}
\item the components  $D_i$ of the negative part $N(\alpha)$ are hypersurfaces invariant by the foliation;  in particular,  $Z(\alpha)$ can be represented by an $\F$ invariant closed positive current.

\item $Z(\alpha)$ is a multiple of $\{\eta_T\}$.

\item  $Z(\alpha)$ is {\it nef} 
 and ${Z(\alpha)}^2=0$.

\item The decomposition is orthogonal: $\{N(\alpha)\} Z(\alpha)=0$. More precisely, for every component $D_i$ of $N(\alpha)$, on a $ \{D_i\} Z(\alpha)=0$.
\item  In $H^{1,1}(M,\R)$, the  $\R$ vector space spanned by the components $\{D_i\}$ of $\{N(\alpha)\}$ intersects the real line spanned by $\{\eta_T\}$ only at the origin.
\item  The decomposition is rational if $X$ is projective and $\alpha$ is a rational class (e.g: $\alpha=c_1(N_{\mathcal F}^*)+\{H\} $)\label{rationaldecomposition}.
\item Every  $(1,1)$ closed positive current wich represents $\alpha$ is necessarily $\mathcal F$-invariant.
\item\label{familleexcep} let $A$ be a hypersurface invariant by the foliation and $A_1,...,A_r$ its irreducible components; the family $\{A_1,...,A_r\}$ is exceptional if and only if the matrix $(m_{ij})=(\{A_i\}\{A_j\}{\{\theta\}}^{n-2})$ is negative , where $\theta$ is a K\" ahler form on $X$.
 \end{enumerate}       
  \end{prop}
Doing the same than in the proof of proposition \ref{transmetricklt}, one can prove that there exists a closed positive current $T$ representing $c_1(\con \otimes\mathcal O (H))$ such that
\begin{equation}\label{hyperboliclogcan}
 T=[N]+\eta_T
\end{equation}     
 As previously, this means that $N_F$ comes equipped with a transverse hyperbolic metric with degeneracies on $K=\mbox{supp}\ N\cup H_b$. This provides a representation
    $$r:\pi_1(X\setminus K)\rightarrow \mbox{Aut}(\mathbb D).$$   
 
Write $N$ as a positive linear combination of prime divisors
$$N=\sum_{i=1}^ p N_i$$
 As previously, one denotes by $\Ihyp$ the corresponding locally constant sheaf of distinguished first integrals .
Like in the previous cases, there are multivaluate sections of $\Ihyp$ around $\mbox{Supp}\ N$  given by elementary first integrals $f_1^{\mu_1}....f_p^{\mu_p}$ with $f_i$ a local section of $\mathcal O(-N_i)$ and $\mu_i$ non negative real number such that $\mu_i=\lambda_i+1$ if $N_i\leq H$ and $\mu_i=\lambda_i$ otherwise (see item (1) of proposition \ref{logformnearK}). By the point (\ref{rationaldecomposition}) of the  proposition above, the exponents $\mu_i$'s are indeed rational whenever $X$ is projective 
In particular, the  monodromy around the local branches of $\mbox{Supp}\ N$ is finite abelian.

It remains to analyse the behaviour of such first integral near the boundary component $H_b$. Roughly speaking, one may think of the leaves space of $\F$ as a (non Hausdorff) orbifold hyperbolic Riemann surface  with finitely many orbifolds points (corresponding to the components of the negative part) and  cusps (corresponding to the components of $H_b$). This is made more precise by the following statement:

\begin{lemma}
 Let $\varphi\in PSL(2,\C)$ an homography such that $\varphi(\mathbb D)=\mathbb H$. Then, around each $x\in H_b$, there exists a local coordinate system $z=(z_1,...,z_n)$ and positive real numbers $\alpha_1,...,\alpha_p$ such that $H_b$ is defined by $z_1....z_p=0$ and 

$$f(z)=-i(\sum_{j=1}^p\alpha_j\log{z_j})$$

is a multivaluate section of $\varphi_* (\Ihyp)$.
In particular, the foliation is defined on a neighborhood of $x$ by the logarithmic form with positive redidues
\begin{equation}\label{canonicalformnearHb}
\omega=\sum_{i=1}^p\alpha_i\frac{dz_i}{z_i}
\end{equation}
\end{lemma}

\begin{proof}
 Assume firstly that $x\notin\mbox{Sing}\ \F$.

The identity (\ref{hyperboliclogcan}) can be locally expressed by an equation involving only one variable $z$ (wich parametrizes the local leaves space and such that $H_b=\{z=0\}$), namely

$$\Delta \varphi (z)=\frac{e^{2\varphi}}{{|z|}^2}.$$

\noindent which is nothing but the equation of negative constant curvature metric near the cuspidal point $z=0$. In particular, the foliation admits as first integrals multivaluate sections of $\phi (\Ihyp)$ of the form $f(z)=-i\alpha \log z + h(z)$ where $h$ is holomrphic and $\alpha >0$. Up right composition by a diffeomorphism of $\C,0$, one can then assume that 

\begin{equation}
 f(z)= -i\alpha \log z
\end{equation}
 (the transverse invariant hyperbolic metric is then given by $\eta_T=\frac{i}{\pi} \frac{dz\wedge d\overline{z}}{{|z|}^2{(\log|z|) }^2}$).
 The general case easily follows from the abelianity of the local monodromy on the complement of the normal crossing divisor $\{z_1...z_p=0\}$.

\end{proof}

\begin{lemma}
The representation $\rho$ has dense image.
\end{lemma}

\begin{proof}
   Let $\omega$ be the logarithmic form given by 
(\ref{canonicalformnearHb}). Note that $\omega$ is nothing but the meromorphic extension of a section of $-id(\varphi_*(\Ihyp))$ near $H_b$ and is uniquely defined up to multiplication by a real positive scalar. 

Consequently, the monodromy representation $\varphi_*r$ restricted to $U\setminus \mathcal C$, where $U$ is a small neighborhood of a connected component $\mathcal C$ of $H_b$ takes values in the affine group $\mbox{Aff}$ of $PSL(2,\R)=\mbox{Aut}(\mathbb H)$ and the image $G_\mathcal C$ contains non trivial tranlations (thanks to the existence of non trivial residues). Moreover, $G_\mathcal C$ does not contain only translations, otherwise this would mean that $\omega$ extends as a logarithmic form on the whole $U$ with poles on $\mathcal C$ and this obviously contradicts lemma \ref{negativeintersectionform}. Denote by $G$ the image of $\varphi_* r$. The observations easily implies the following alternative

\begin {enumerate}
 \item Either $G\subset \mbox{Aff}$, 
\item either the topological closure $\overline{G}$ of $G$ is $PSL(2,\R)$.
\end {enumerate}
By an  argument already used in section \ref{sectiondynamics}, the first case forces the positive part $Z$ to vanish and then yields a contradiction.

The monodromy group of the foliation is thus dense in $\mbox{Aut}(\mathbb D)$. By using a bundle-metric complete at infinity in the spirit of lemma \ref{bundlelikemetricatinfinity} (here, it is convenient to choose local model of the form $g_x=i\sum_{l=1}^{n_x}\frac{{dz_l}\wedge\{d\overline{z_l}}{{|z_l|}^2\{(\log{|z_l|)}^2} +h_x$ near $x\in H_b$). Following the same line of argumentation as in \ref{caseZ=0}, one can show that the dynamical behaviour of $\F$ is completely encoded by $G$ and in particular that $\F$ is a {\bf quasi-minimal foliation}.

\end{proof}

\begin{remark}
As an illustration of this monodromy behavior, it may be relevant to keep in mind the example of modular Hilbert modular foliations (see \cite{pemen} for a thorough discussion on this subject). In this case, one may think of $H_b$ as the exceptional divisors arising from the desingularization of cusps and the affine monodromy near $H_b$ is here related to the isotropy groups of the cusps (the maximal parabolic subgroup of the lattice defining the Hilbert modular surface).
\end{remark}

 \begin{remark}
 Like in the non logarithmic and KLT case, the local monodromy at infinity remains quasi-unipotent . The new phenomenon, here, is that it is no more finite monodromy around the components of $H_b$ (because of the $\log$ terms in the local expression of distinguished first integrals).
 \end{remark}
 
\subsubsection{Proof of the main theorem}
It is namely the theorem \ref{thprincipallog} of the introduction.
Recall that $\F$ is a strict Log-Canonical foliation with non vanishing positive part $Z$. From now on, we will assume that the ambient manifold $X$ is {\bf projective}. 

Thanks to the transverse hyperbolic structure together with the quasi-minimality of $\F$, one can conclude by the use of 
the Schwarzian derivative "trick" (see the proof of theorem \ref{abundancecases}) that  $\kappa:=\mbox{Kod}(N_{\mathcal F}^*\otimes \mathcal O(H))=-\infty$. Indeed, if $\kappa\geq 0$,  up doing a suitable ramified covering, one can assume that the foliation is defined by a logarithmic form $\omega$ with normal crossing poles, hence closed.

Thanks to the moderate growth of distinguished first integrals,  $\F$ is a transversely projective foliations with regular singularities and the corresponding Riccati foliation does not factors through a Ricatti foliation over a curve for exactly the same dynamical reasons as those already observed  in \ref{projectivetriple}. 

One can then similarly claim that 
there exists 
 $\Psi$  a morphism of analytic varieties between $X\setminus H_b$ and the quotient $\frak{H}={{\D}^n}/{\Gamma}$ of a polydisk ($n\geq 2)$ by an irreducible lattice $\Gamma\subset {(\mbox{Aut}\ \D)}^n$ such that $\mathcal F={\Psi}^*\mathcal G$ where $\mathcal G$ is one of the $n$ tautological foliations on $\frak{H}$.

 To get the final conclusion, one extends $\Psi$ to an algebraic morphism 
 
         $$\overline {\Psi}\rightarrow {\overline{\frak H}}^{BB}$$
         
         thanks to Borel's extension theorem \cite{bor}.



\end{document}